\documentclass[11pt,reqno]{amsart}
\usepackage{color}
\usepackage{amsmath}
\usepackage{amssymb}
\usepackage{amsfonts}

\usepackage[lmargin=3cm,rmargin=3cm,tmargin=3cm,bmargin=3cm]{geometry}

\definecolor{darkviolet}{rgb}{0.58, 0.0, 0.83}
\def\R{\mathbb{R}}

\numberwithin{equation}{section}

\newcommand{\EP}{\exp L^2}
\newcommand{\EPO}{\exp L_0^2}

\newtheorem{theorem}{Theorem}[section]
\newtheorem{lemma}[theorem]{Lemma}
\newtheorem{proposition}[theorem]{Proposition}
\newtheorem{corollary}[theorem]{Corollary}

\newtheorem{definition}[theorem]{Definition}

\newtheorem{remark}[theorem]{Remark}
\newtheorem{remarks}[theorem]{Remarks}

\author[M. Majdoub]{Mohamed Majdoub\textsuperscript{1}}
\author[S. Otsmane]{Sarah Otsmane\textsuperscript{2}}
\author[S. Tayachi]{Slim Tayachi\textsuperscript{2}}

\title[The Biharmonic heat equation with exponential nonlinearity]
{Local well-posedness and global existence for the biharmonic heat equation with exponential nonlinearity}
\subjclass[2010]{35K91, 35A01, 35B40, 35K25, 35K30, 46E30}
\keywords{Biharmonic heat equation, Exponential nonlinearity, Orlicz spaces, Well-posedness, Global existence, Decay estimates.}
\begin{document}
\maketitle
\begin{center}
{\footnotesize \textsuperscript{1} Imam Abdulrahman Bin Faisal University, College of Science, Mathematics Department, Dammam, KSA, E-mail address: mmajdoub@uod.edu.sa}

{\footnotesize \textsuperscript{2} Universit\'e de Tunis El Manar, Facult\'e des Sciences de Tunis, D\'epartement de Math\'ematiques, Laboratoire \'equations aux d\'eriv\'ees partielles (LR03ES04), 2092 Tunis,
Tunisie,\\
E-mail addresses: sarah.osmane4@gmail.com (S. Otsmane),  slim.tayachi@fst.rnu.tn  (S. Tayachi)}
\end{center}
\begin{abstract} In this paper we prove local well-posedness in Orlicz spaces for the biharmonic heat equation $\partial_{t} u+ \Delta^2 u=f(u),\;t>0,\;x\in\R^N,$ with $f(u)\sim \mbox{e}^{u^2}$ for large $u.$ Under smallness condition on the initial data and for exponential nonlinearity $f$ such that  $|f(u)|\sim |u|^m$ as $u\to 0,$  $m\geq 2$, $N(m-1)/4\geq 2$, we show that the solution is global. Moreover, we obtain decay estimates for large time for the nonlinear biharmonic heat equation as well as for the nonlinear heat equation. Our results extend to the nonlinear polyharmonic heat equation.
\end{abstract}


\section{Introduction}


In this paper we study the local well-posedness and the global existence of solution to the Cauchy problem:
\begin{equation}\label{1.1}
\left\{\begin{array}{cc}
\partial_{t} u+ \Delta^2 u=f(u)  ,   \\
u(0,x)=u_{0}(x),
\end{array}
\right.
\end{equation}
where $u(t,x)$ is a real valued function  $t>0,\,\,x\in\R^N,$ $\Delta^2$ is the biharmonic operator, and $f : \R\to\R$  having an exponential growth at infinity.

Higher order parabolic equations, in particular fourth order parabolic equations are used in many models. They arise in Cahn-Hillard equations, image segmentation, epitaxial thin film growth,
surface diffusion flow equations, biharmonic heat equation, etc.
See \cite{KingWinkler,Vertman,BellShubinStephens,AsaiGiga,Escudero} and references therein.
In the literature no so much is known on higher order parabolic equations, compared with the second order ones.
This is due in particular to the lack of the maximum principle.
In this paper, we study the local existence and asymptotic behavior for the biharmonic heat equation with initial data in Orlicz spaces and with nonlinearity behaving as $e^{u^2}$ at infinity.
We also obtain some results for the higher order nonlinear heat equation.

As is a standard practice, we study \eqref{1.1} via the associated integral equation:
\begin{equation}
\label{integral}
u(t)= {\rm e}^{-t\Delta^2}u_{0}+\int_{0}^{t}{\rm e}^{-(t-s)\Delta^2}\,f(u(s))\,ds,
\end{equation}
where ${\rm e}^{-t\Delta^2}$ is the biharmonic heat semi-group.

It is well known that if $u_0\in L^\infty(\R^N)$ and $f(0)=0$, then there exists a unique local solution $u\in C((0,T];\,L^\infty(\R^N))$ to \eqref{1.1}. If $u_0\not\in L^\infty(\R^N)$, the first result with singular initial
data is due to Weissler \cite{Indiana} for power nonlinearities $f(u)=|u|^{m-1}\,u,$ $m>1$ is a real number and $u_0\in L^q(\R^N),$\, $1\leq q<\infty$. From \cite[Theorem 2, p. 82]{Indiana} with $A=-\Delta^2$,
we have local well-posedness in $L^q(\R^N)$ for $q>q_c:=\frac{N(m-1)}{4},\; q\geq 1.$

For exponential nonlinearities $f(u)\sim {\rm e}^{u^2}$ , $u$ large, the better space is the so-called Orlicz space  $\exp L^2(\R^N)$ which is a generalization of Lebesgue spaces and is
embedded in $L^q(\R^N)$ for every $2\leq q<\infty.$ While the second order heat equation with exponential nonlinearity has been studied extensively before, see \cite{IJMS, Ioku, IRT, RT} and references therein, to our best  knowledge there is no previous study on the biharmonic heat equation with a such nonlinearity.
This motivates us to consider the problem \eqref{1.1} with exponential growth nonlinearity.

The Orlicz space $\exp L^2(\R^N)$ is defined as follows
\begin{equation*}
    \exp L^2(\R^N)=\bigg\{\,u\in L^1_{loc}(\R^N);\;\int_{\R^N}\Big({\rm e}^{|u(x)|^2\over\alpha^2}-1\Big)\,dx<\infty ,\,\;\mbox{for some}\,\; \alpha>0\, \bigg\},
\end{equation*}
endowed with the Luxembourg norm
\begin{equation*}
    \|u\|_{\exp L^2(\R^N)}:=\inf\biggr\{\,\alpha>0;\,\,\,\,\int_{\R^N} \Big({\rm e}^{|u(x)|^2\over\alpha^2}-1\Big)\,dx\leq1\,\biggl\}.
\end{equation*}
Since the space of smooth compactly supported functions $C^{\infty}_0(\R^N)$ is not dense in the Orlicz space $\exp L^2(\R^N)$ (see \cite{Ioku, IRT}), we use the space $\exp L^2_0(\R^N)$  which is the closure
of $C^{\infty}_0(\R^N)$ with respect to  the Luxemburg norm $\|\cdot\|_{\exp L^2(\R^N)}$.
It is known that, see \cite{IRT},
\begin{equation}\label{exp L0}
    \exp L^2_0(\R^N)=\biggr\{\,u\in L^1_{loc}(\R^N);\;\int_{\R^N}\Big({{\rm e}^{|u(x)|^2\over\alpha^2}-1\Big)}\,dx<\infty ,\,\; \mbox{for every}\,\; \alpha>0\, \biggl\}.
\end{equation}
As it will be shown in Section 3, the biharmonic heat semi-group  ${\rm e}^{-t\Delta^2}$ is continuous at $t=0$ in $\exp L^2_0(\R^N).$ However, this is not the case in  $\exp L^2(\R^N).$

In the sequel, we adopt the following definitions of weak and weak-mild solutions to the Cauchy problem \eqref{1.1}.
\begin{definition}[{Weak solution}]
\label{weak}
Let $u_0\in\exp L^2_0(\R^N)$ and $T>0$. We say that the function $u\in C([0,T];\,\exp L^2_0(\R^N))$ is a weak solution of \eqref{1.1} if $u$ verifies \eqref{1.1} in the sense of distribution and $u(t)\to u_0$ in the weak$^*$topology as $t\searrow 0.$
\end{definition}

\begin{definition}[{Weak-mild solution}]
\label{weakmild}
We say that $u\in L^{\infty}(0,T;\, \exp L^2(\R^N))$  is a weak-mild solution of
the Cauchy problem \eqref{1.1} if $u$ satisfies the associated integral equation \eqref{integral} in $\exp L^2(\R^N)$ for almost all $t\in (0,T)$ and $u(t)\rightarrow u_0$ in the weak$^*$ topology as $t\searrow 0.$
\end{definition}

We are first interested in the local well-posedness. Since $C_0^\infty(\R^N)$ is dense in $\exp L^2_0(\R^N)$, we are able to prove local existence and uniqueness to \eqref{1.1}  for initial data in $\exp L^2_0(\R^N)$. We assume that the nonlinearity $f$ satisfies
\begin{equation}\label{1.9}
    f(0)=0,\qquad |f(u)-f(v)|\leq C|u-v|({\rm e}^{\lambda \,u^2}+{\rm e}^{\lambda \,v^2}),\;\forall\;\;u,\,v\in\R,
\end{equation}
for some constants $C>0$ and $\lambda>0$. Our first main result reads as follows.
\begin{theorem}[{Local well-posedness}]\label{local} Suppose that $f$ satisfies \eqref{1.9}. Given any $u_0\in\exp L^2_0(\R^N),$ there exist a time $T=T(u_0)>0$ and a unique weak solution $u\in C ([0,T];\,\exp L^2_0(\R^N))$ to \eqref{1.1}.
\end{theorem}

Our second interest is the global existence. This depends on the behavior of the nonlinearity $f(u)$ near $u=0.$ The following behavior near $0$ will be allowed
$$|f(u)|\sim |u|^m,$$
where $m\geq 2$, $N(m-1)/4\geq 2.$  In fact, by Fujita-type results, it is known that for $N(m-1)/4<1$ blow-up occurs. See \cite{egorov, GalakMitidieri} and references therein. In addition,  in the framework of  $\exp L^2(\R^N)$ and using the embedding of Orlicz space in Lebesgue space (see Lemma \ref{sarah55} below), imposes that $N(m-1)/4\geq 2.$ More precisely, we suppose that the nonlinearity $f$ satisfies
\begin{equation}\label{1.4}
f(0)=0,\qquad|f(u)-f(v)|\leq C\left|u-v\right|\bigg(|u|^{m-1}{\rm e}^{\lambda u^2}+|v|^{m-1}{\rm e}^{\lambda v^2}\bigg),\quad\forall\;u,\;v\in\R,
\end{equation}
where  $m\geq 1+{8\over N}$, $m\geq 2$, $C>0$ and $\lambda>0$ are constants. Our aim is to obtain global existence to the Cauchy problem \eqref{1.1} for small initial data in $\exp L^2(\R^N)$. We have obtained the following.
\begin{theorem}[{Global existence}]
\label{GE} Let $m$ be a real number such that $m\geq 1+{8\over N},$ $m\geq 2.$ Assume that the nonlinearity $f$ satisfies \eqref{1.4}. Then, there exists a positive constant $\varepsilon>0$ such that for every initial data $u_0\in \exp L^2(\R^N)$
with $\|u_0\|_{\exp L^2(\R^N)} \leqslant\varepsilon,$ there exists a weak-mild solution $u\in L^{\infty}(0,\infty;\exp L^2(\R^N))$
of the Cauchy problem \eqref{1.1} satisfying
\begin{equation} \label{1.8}
    \lim_{t\longrightarrow0}\|u(t)-{\rm e}^{-t\Delta^2}u_{0}\|_{\exp L^2(\R^N)}=0.
\end{equation}
Moreover, there exists a constant $C>0$ such that,
\begin{equation}
\label{Linfinibihar}
\|u(t)\|_p \leq\,C\,t^{-\sigma},\quad\forall\;t>0,
\end{equation}
where $$\sigma={1\over m-1}-{N\over 4p}>0,$$with ${N(m-1)\over 4}<p<\infty$ if $N\geq 8$ and $\max\left\{m,\;{2N(m-1)\over 8-N}\right\}\leq p\leq \infty$ if $N<8$.
\end{theorem}

Hereafter, $\|\cdot\|_p$ denotes the norm in the Lebesgue space $L^p(\R^N),\; 1\leq p\leq \infty.$ We mention that the assumption \eqref{1.4} on  the nonlinearity covers the cases
$$
f(u)=\pm |u|^{m-1}u\,{\rm e}^{u^2},\quad m\geq 1+{8\over N},\;m\geq 2.
$$

Let us now give an outline of the proofs. To prove Theorem \ref{local}, we decompose the initial data $u_0\in\exp L^2_0 (\R^N)$ into a small part in $\exp L^2(\R^N)$ and a smooth one.
This will be done using the density of $C^\infty_0(\R^N)$ in $\exp L^2_0(\R^N)$. First we solve the initial value problem with smooth initial data to obtain a local and bounded solution $v$.
Then we consider the perturbed equation satisfied by $w := u-v$ and with small initial data. This idea was introduced in \cite{Bou} and used in \cite{Kenig, APDE} for instance.
Our proof of Theorem \ref{GE} uses and improves some arguments of \cite{Ioku} and is done as follows. First, we establish a key $L^p-L^q$ estimate  on the biharmonic heat semi-group with a constant independent of $p$ and $q$ (see Proposition \ref{LPLQQ} below). This is obtained using a majorizing kernel of the fundamental solution of the linear biharmonic heat equation, since this last one is not positive. Then, we perform a fixed point argument in suitable complete metric space. This space allows us to obtain the estimate \eqref{Linfinibihar} immediately.

Using this approach, we can obtain similar estimates for the nonlinear heat equation
\begin{equation}\label{H}
\left\{\begin{array}{cc}
\partial_{t} u-\Delta u=f(u)  ,   \\
u(0,x)=u_{0}(x),
\end{array}
\right.
\end{equation}
where $u=u(t,x),\,t>0,\,x\in\R^N$ and $f$ satisfies
\begin{equation}\label{HH}
f(0)=0,\qquad |f(u)-f(v)|\leq C\left|u-v\right|\bigg(|u|^{\ell-1}{\rm e}^{\lambda u^2}+|v|^{\ell-1}{\rm e}^{\lambda v^2}\bigg),\quad\forall\;u,\;v\in\R,
\end{equation}
with ${N(\ell-1)\over 2}\geq 2$, $\ell\geq 2$, $C>0$ and $\lambda>0$ are constants.
It is known that if $u_0\in\exp L^2_0(\R^N)$, then \eqref{H} has a unique local solution (see \cite{IRT}).
It is also known that if  $\|u_0\|_{\exp L^2(\R^N)}$ is sufficiently small, the solution is global (see \cite{Ioku}).
We have the following result which can be seen as an improvement of \cite[Theorem 1.3, p. 1174]{Ioku}. In fact, in the following theorem, only the case $\ell=1+4/N,\; N\leq 4$ was considered in \cite{Ioku}.
\begin{theorem}\label{linfinheat}
Let $\ell$ be a real number such that $\ell\geq 1+{4\over N},$ $\ell\geq 2.$ Assume that the nonlinearity $f$ satisfies \eqref{HH}. Let $u_0\in\exp L^2_0(\R^N)$ with $\|u_0\|_{\exp L^2(\R^N)}$ sufficiently small, and $u\in C([0,\infty);\,\exp L^2_0(\R^N))$ be the global solution of \eqref{H}. Then we have the following
\begin{equation}
\label{Linfiniheat}
\|u(t)\|_p\leq\,C\,t^{-\sigma},\quad\forall\;t>0,
\end{equation}
where $$\sigma={1\over \ell-1}-{N\over 2p}>0,$$
 with  ${N(\ell-1)\over 2}<p<\infty$ if $N\geq 4$ and $\max\left\{\ell,\;{2N(\ell-1)\over 4-N}\right\}\leq p\leq \infty$ if $N<4$.
\end{theorem}

The rest of this paper is organized as follows. In the next section, we collect some basic facts and useful tools about Orlicz spaces.  Section 3 is devoted to some crucial estimates on the biharmonic heat semi-group. The proof of Theorem \ref{local} is done in Section 4. In Section 5, we give the proof of Theorem \ref{GE}. We also give the proof of Theorem \ref{linfinheat} concerning the nonlinear heat equation \eqref{H}. Finally, in Section 6 we sketch briefly  how our results can be extended to { polyharmonic} heat equations with exponential nonlinearity.  In all this paper, $C$ will be a positive constant which  may have different values at different places. Also, $L^q(\R^N)$, $\exp L^2(\R^N)$, $\exp L^2_0(\R^N)$ will be written respectively $L^q$, $\EP$ and $\EPO$.

\section{Orlicz spaces: basic facts and useful tools}
Let us recall the definition of the so-called Orlicz spaces on $\R^N$ and some
related basic facts. For a complete presentation and more details, we refer the reader to
\cite{Adams, RR, Trud}.

\begin{definition}\label{deforl}\quad\\
Let $\phi : \R^+\to\R^+$ be a convex increasing function such that
$$
\phi(0)=0=\lim_{s\to 0^+}\,\phi(s),\quad
\lim_{s\to\infty}\,\phi(s)=\infty.
$$
We say that a  function $u\in L^1_{loc}(\R^N)$ belongs to
$L^\phi(\R^N)$ if there exists $\alpha>0$ such that
$$
\int_{\R^N}\,\phi\left(\frac{|u(x)|}{\alpha}\right)\,dx<\infty.
$$
We denote then
\begin{equation}
\label{Luxemb}
\|u\|_{L^\phi}=\inf\,\left\{\,\alpha>0,\quad\int_{\R^N}\,\phi\left(\frac{|u(x)|}{\alpha}\right)\,dx\leq
1\,\right\}.
\end{equation}
\end{definition}
It is known that $\left(L^\phi(\R^N),\|\cdot\|_{L^\phi}\right)$ is a Banach space. Note that, if $\phi(s)=s^p,\, 1\leq
p<\infty$,  then $L^\phi$ is nothing else than the  Lebesgue space $L^p$. Moreover, for $u\in L^\phi$ with $\|u\|_{L^\phi}>0$,
we have
$$\left\{\,\alpha>0,\quad\int_{\R^N}\,\phi\left(\frac{|u(x)|}{\alpha}\right)\,dx\leq
1\,\right\}=[\|u\|_{L^\phi}, \infty[\,. $$
In particular
\begin{equation}
\label{med1}
\int_{\R^N}\,\phi\left(\frac{|u(x)|}{\|u\|_{L^\phi}}\right)\,dx\leq 1.
\end{equation}
We also recall the following well known properties.
\begin{proposition} \label{fatou} We have
\begin{itemize}
\item[(i)] $L^1\cap L^\infty\subset L^\phi\subset L^1+L^\infty$.
\item[(ii)] {\it Lower semi-continuity}:
$$
u_n\to u\quad\mbox{a.e.}\quad\Longrightarrow\quad\|u\|_{L^\phi}\leq
\liminf\|u_n\|_{L^\phi}.
$$
\item[(iii)] {\it Monotonicity}:
$$
|u|\leq|v|\quad\mbox{
a.e.}\quad\Longrightarrow\quad\|u\|_{L^\phi}\leq\|v\|_{L^\phi}.
$$
\item[(iv)] {\it Strong Fatou property}:
$$
0\leq u_n\nearrow u\quad\mbox{
a.e.}\quad\Longrightarrow\quad\|u_n\|_{L^\phi}\nearrow\|u\|_{L^\phi}.
$$
\end{itemize}
\end{proposition}

Denote by
$$
L_0^\phi(\R^N)=\bigg\{\,u\in L^1_{loc}(\R^N),\;\;\; \int_{\R^N}\,\phi\left(\frac{|u(x)|}{\alpha}\right)\,dx<\infty,\;\forall\;\;\alpha>0\bigg\}\,.
$$

It can be shown (see for example \cite{IRT}) that
$$
 L_0^\phi(\R^N)=\overline{C_0^\infty(\R^N)}^{L^\phi}= \mbox{the {closure} of $C_0^\infty(\R^N)$ in $L^\phi(\R^N)$}.
 $$
Clearly $L_0^\phi(\R^N)=L^\phi(\R^N)$ for $\phi(s)=s^p, \,p\geq 1$, but this is not the case for any $\phi$ (see \cite{IRT}). When $\phi(s)={\rm e}^{s^2}-1$, we denote the space $L^\phi(\R^N)$ by $\EP$ and $L^\phi_0(\R^N)$ by $\EPO$.

{ The following Lemma summarize the relationship between Orlicz and Lebesgue spaces.
\vspace{-0.3cm}
\begin{lemma}\label{Orl-Leb}  We have
\begin{itemize}
\item[(i)] $\EPO\varsubsetneq \EP $.
\item[(ii)] $\EPO\not\hookrightarrow L^\infty$, hence\;$\EP\not\hookrightarrow L^\infty.$
\item[(iii)] $\EP \not\hookrightarrow L^r$,\;\; for all\;\;$1\leq r<2$.
\item[(iv)] $L^2\cap L^\infty\hookrightarrow \EPO$. More precisely
\begin{equation}
\label{log2}
\|u\|_{\EP}\leq \frac{1}{\sqrt{\log 2}}\,\bigg(\|u\|_2+\|u\|_\infty\bigg).
\end{equation}
\end{itemize}
\end{lemma}
\begin{proof}\quad\\
\vspace{-0.7cm}
\begin{itemize}
\item[(i)] Let $u$ be the function defined by
\begin{eqnarray*}
u(x)&=&\bigg(-\log|x|\bigg)^{1/2}\quad\mbox{if}\quad |x|\leq 1,\\
u(x)&=&0\qquad\mbox{if}\qquad |x|>1.
\end{eqnarray*}
For $\alpha>0$, we have
$$
\int_{\R^N}\,\bigg({\rm e}^{\frac{u(x)^2}{\alpha^2}}-1\bigg)\,dx<\infty\Longleftrightarrow \alpha>N^{-1/2}.
$$
Therefore $u\in \EP\backslash\EPO$.
\item[(ii)]  Let $u$ be the function defined by
\begin{eqnarray*}
u(x)&=&\bigg(\log\left(1-\log|x|\right)\bigg)^{1/2}\quad\mbox{if}\quad |x|\leq 1,\\
u(x)&=&0\qquad\mbox{if}\qquad |x|>1.
\end{eqnarray*}
Clearly $u\not\in L^\infty$. Moreover, for any $\alpha>0$, we have
$$
\int_{\R^N}\,\bigg({\rm e}^{\frac{u(x)^2}{\alpha^2}}-1\bigg)\,dx=|\mathcal{S}^{N-1}|\int_0^1\, r^{N-1}\bigg((1-\log r)^{1\over \alpha^2} -1\bigg)\,dr <\infty,
$$
where $|\mathcal{S}^{N-1}|$ is the measure of the unit sphere $\mathcal{S}^{N-1}$ in $\R^N$. The second assertion follows since  $\EPO \hookrightarrow\EP $.
\item[(iii)] Let $u$ be the function defined by
\begin{eqnarray*}
u(x)&=&|x|^{-{N\over r}}\quad\mbox{if}\quad |x|\geq 1,\\
u(x)&=&0\qquad\mbox{if}\qquad |x|<1.
\end{eqnarray*}
Then $u\in \EPO$ but $u\not\in L^r$. Indeed, it is clear that $u\not\in L^r$, and for $\alpha>0$, we have
$$
\int_{\R^N}\,\bigg({\rm e}^{\frac{u(x)^2}{\alpha^2}}-1\bigg)\,dx=\frac{N|\mathcal{S}^{N-1}|}{r}\,\sum_{k=1}^\infty\, \frac{1}{(2k-r)k!\alpha^{2k}}<\infty.
$$
\item[(iv)] Let $u\in L^2\cap L^\infty$ and let $\alpha>0$. Using the interpolation inequality
$$
\|u\|_{q}\leq \|u\|_{2}^{2/q}\,\|u\|_{\infty}^{1-{2/q}}\leq\|u\|_{2}+\|u\|_{\infty} ,\quad 2\leq q<\infty,
$$
we obtain
\begin{eqnarray*}
\int_{\R^N}\,\bigg({\rm e}^{\frac{u(x)^2}{\alpha^2}}-1\bigg)\,dx&=&\sum_{k=1}^\infty\,\frac{1}{k!\alpha^{2k}}\|u\|_{L^{2k}}^{2k}\\
&\leq& \sum_{k=1}^\infty\,\frac{1}{k!\alpha^{2k}}\left(\|u\|_{2}+\|u\|_{\infty}\right)^{2k}\\
&\leq&  {\rm e}^{\frac{\left(\|u\|_{2}+\|u\|_{\infty}\right)^2}{\alpha^2}}-1.
\end{eqnarray*}
This clearly implies \eqref{log2}.
\end{itemize}
\end{proof}

}

 We have the embedding: $\EP \hookrightarrow L^r$ for every $2\leq r<\infty$. More precisely:
\begin{lemma}({\cite{RT}})\label{sarah55} For every $2\leq r<\infty,$ we have
\begin{equation}\label{2.1}
    \|u\|_r\leqslant\left(\Gamma\left(\frac{r}{2}+1\right)\right)^\frac{1}{r}\|u\|_{\EP},
  \end{equation}
where $\Gamma(x):=\displaystyle\int_0^{\infty}\tau^{x-1}{\rm e}^{-\tau}\,d\tau, \; x>0.$
\end{lemma}

The following lemma will be useful in the { proofs.}
\begin{lemma}
\label{med}
Let $\lambda>0$, $1\leq p<\infty$ and ${ K>0}$ such that $\lambda\,p\,{K}^2\leq 1$. Assume that
$$
\|u\|_{\EP}\leq {K}\,.
$$
Then
$$
\|{\rm e}^{\lambda u^2}-1\|_p\leq \left(\lambda\,p\,{K}^2\right)^{{1\over p}}\,.
$$
\end{lemma}
\begin{proof}
Write
\begin{eqnarray*}
\int_{\R^N}\,\left({\rm e}^{\lambda u^2}-1\right)^p\,dx &\leq& \int_{\R^N}\,\left({\rm e}^{\lambda\,p\,u^2}-1\right)\,dx\\
&\leq&\int_{\R^N}\,\left({\rm e}^{\lambda\,p\,{K}^2\frac{u^2}{\|u\|_{\exp L^2}^2}}-1\right)\,dx\\
&\leq&\,\lambda\,p\,{K}^2\,\int_{\R^N}\,\left({\rm e}^{\frac{u^2}{\|u\|_{\exp L^2}^2}}-1\right)\,dx\leq \,\lambda\,p\,{K}^2,
\end{eqnarray*}
where we have used the fact that ${\rm e}^{\theta s}-1\leq \theta\left({\rm e}^{s}-1\right)$, $0\leq \theta\leq 1$, $s\geq 0$ and \eqref{med1}.
\end{proof}

We state the following proposition which is needed for the local well-posedness in the space $\EPO$.
\begin{proposition}\label{step1} Let $u\in C([0,T]; \EPO)$. Assume that $f$ satisfies \eqref{1.9}. Then for every $2\leq p<\infty$ there holds
\begin{equation*}
    f(u)\in C\left([0, T]; L^p\right).
\end{equation*}
\end{proposition}
\begin{proof}[Proof of Proposition \ref{step1}]
Fix $2\leq p<\infty$, $0\leq t\leq T$ and let $(t_n)\subset [0,T]$ such that $t_n\to t$. Using H\"older inequality, we obtain
\begin{eqnarray*}
\|f(u(t_n))-f(u(t))\|_p&\leq & 2C\|u(t_n)-u(t)\|_p+\\
&&C\||u(t_n)-u(t)|({\rm e}^{\lambda \,u(t_n)^2}-1+{\rm e}^{\lambda \,u(t)^2}-1)\|_p\\
&\leq&  2C\|u(t_n)-u(t)\|_p+C\|u(t_n)-u(t)\|_{{2p}}\times \\
&&\left(\|{\rm e}^{\lambda \,u(t_n)^2}-1\|_{{2p}}+\|{\rm e}^{\lambda \,u(t)^2}-1\|_{{2p}}\right)\\
&&\hspace{-3cm}\leq C\|u(t_n)-u(t)\|_{\EP}\left(1+\|{\rm e}^{\lambda \,u(t_n)^2}-1\|_{{2p}}+\|{\rm e}^{\lambda \,u(t)^2}-1\|_{{2p}}\right),
\end{eqnarray*}
where we have used Lemma \ref{sarah55} in the last inequality. From \cite{IJMS, IRT} (see Proposition 2.3 in \cite{IRT} for instance) we know that $\|{\rm e}^{\lambda \,u(t_n)^2}-1\|_{{2p}}\to \|{\rm e}^{\lambda \,u(t)^2}-1\|_{{2p}}$ as $n\to\infty$. It follows that $\|f(u(t_n))-f(u(t))\|_{p}\to 0$ which is the desired conclusion.
\end{proof}
\begin{remark}
{\rm The assumption $p\geq 2$ is crucial in order to apply Lemma \ref{sarah55}. We believe that the conclusion of Proposition \ref{step1} fails when $1\leq p<2$.}
\end{remark}

We close this section by recalling some properties of the functions $\Gamma$ and ${\mathcal{B}}$ given by
\begin{eqnarray*}
\Gamma(x)&=&\displaystyle\int_0^{\infty}\tau^{x-1}{\rm e}^{-\tau}\,d\tau, \; x>0.\\
 {\mathcal{B}}(x,y)&=&\int_0^1\tau^{x-1}(1-\tau)^{y-1}d\tau,\quad x,\; y>0.
 \end{eqnarray*}
  We have
\begin{equation}
\label{gamma4}
 {\mathcal{B}}(x,y)=\frac{\Gamma(x+y)}{\Gamma(x)\Gamma(y)},\; \forall\; x,\; y>0,
\end{equation}

\begin{equation}
\label{gamma1}
\Gamma(x)\geq C>0,\; \forall\; x>0,
\end{equation}

\begin{equation}
\label{gamma2}
\Gamma(x+1)\sim \left({x\over \rm{e}}\right)^x\, \sqrt{2\pi x}, \; \mbox{ as }\; x \to \infty,
\end{equation}
and
\begin{equation}
\label{gamma3}
\Gamma(x+1)\leq C x^{x+{1\over 2}},\; \forall \; x\geq 1.
\end{equation}

These estimates will be needed in the proof of Theorem \ref{GE}.


\section{Linear estimates}

In this section we establish some results needed for the proofs of the main theorems. We first establish some basic estimates for the linear biharmonic heat semigroup ${\rm e}^{-t\Delta^2}.$   We consider  the problem
\begin{equation}\label{4444444}
\left\{\begin{array}{cc}
\partial_{t} u=- \Delta^2 u,\; t>0,\; x\in \R^N,\\
u(0,x)=u_{0}(x).
\end{array}
\right.
\end{equation}
The solution of \eqref{4444444} can be written as a convolution:
\begin{equation*}
    u(t,x)=\bigr(E_t\star u_0\bigl)(x):=\bigr({\rm e}^{-t\Delta^2}u_0\bigl)(x),
\end{equation*}
where
\begin{equation*}
   E_t(x)=\frac{1}{(2\pi)^N}\int_{\R^N}{\rm e}^{- t|\xi|^4}\,{\rm e}^{{\rm i}x\cdot\xi}\,d\xi,
\end{equation*}
is the biharmonic heat kernel.  Clearly
$$
E_t(x)=t^{-{N\over 4}}\,E_1(t^{-{1\over 4}}x).
$$
It is known that $E_1$ is given by
 $$E_1(x)=|x|^{1-N}\int_0^\infty (|x|s)^{N/2}J_{(N-2)/2}(|x|s){\,{\rm e}^{-s^4}}\,ds,$$
where $J_{\nu}$ denotes the $\nu$-th Bessel function of the first kind. See \cite{Gazz}  and references therein. This implies in particular that the kernel $E$  changes sign, therefore, the
associated semigroup is not order-preserving. This is different from the case of the heat equation. In fact, if $u_0\in C(\R^N)\cap L^\infty(\R^N),$ the unique global solution of \eqref{4444444} can be written as follows
\begin{equation*}
u(t,x)={t^{-\frac{N}{4}}\int_{\R^N} E_1\left(t^{-\frac{1}{4}}y\right)}u_0(x-y)dy.
\end{equation*}

We will frequently use the $L^p-L^q$ estimate as stated in the proposition below.
\begin{proposition}\label{LPLQQ} There exists a positive constant $ \mathcal{H}$ such that for all $1\leq p\leq q\leq\infty$, we have
\begin{equation}\label{x}
\|{\rm e}^{-t\Delta^2}\varphi\|_{q}\leqslant \mathcal{H}  t^{-\frac{N}{4}(\frac{1}{p}-\frac{1}{q})}\|\varphi\|_{p},\qquad\,\, \forall\; t>0,\,\, \forall\; \varphi\in L^p.
\end{equation}
\end{proposition}

\begin{remarks} \quad\\
\vspace{-0.5cm}
{\rm
\begin{itemize}
\item[1)] The fact that the constant $ \mathcal{H}$ does not depend on $p$ and $q$ is crucial in our approach. In fact in \cite{W, SMG} the estimate \eqref{x} was obtained {for $1<p\leq q<\infty$ and } with a constant depending on $p$ and $q.$  Although in these references the estimates are for bounded domains they are valid for the hole space. See \cite[Lemma 4.1, p. 293]{W} and \cite[Lemma 2.1 , p. 7264 ]{SMG}.   This is not helpful for our case and the constant should be independent of $p$ and $q.$
\item[2)] For the standard linear heat  equation, it is known that we may take the constant smaller than $1$ and hence independent of $p$ and $q.$
\item[3)] It seems that for $p=q$ we may take the constant in \eqref{x} smaller than $1.$ See \cite[Formula (2.4)]{IKK}. Since this is not important for our proofs, we do not use it.
\end{itemize}
}
\end{remarks}

To prove Proposition \ref{LPLQQ} we need the following proposition which gives an exponential decay of the biharmonic heat kernel.

\begin{proposition}(\cite[Proposition 2.1]{GP}) \label{abir2} There exist two constants $\mathcal{K} > 1$ and $\mu>0$ such that
\begin{equation*}
    |E_1(\eta)|\leq \mathcal{K}\, F(\eta), \; \forall\; \eta\in \R^N,
\end{equation*}
where
$$F(\eta)=\omega {\rm e}^{-\mu|\eta|^{4/3}}, \, \eta \in\,\, \R^N,\quad \omega^{-1}=\int_{\R^N}{\rm e}^{-\mu|\eta|^{4/3}}d\,\eta.
$$
\end{proposition}

\begin{proof}[Proof of Proposition \ref{LPLQQ}] Let $1\leq p\leq q \leq \infty,\; t>0$ and $\varphi\in L^p.$ Using Young inequality, we have
$$\|{\rm e}^{-t\Delta^2}\varphi\|_{q}=\|E_t{\star }\varphi\|_{q}\leq \|E_t\|_{r}\|\varphi\|_{p},$$
where
$$\frac{1}{q}+1=\frac{1}{p}+\frac{1}{r}.$$
Since
$E_t(x)=t^{-\frac{N}{4}}E_1\left(t^{-\frac{1}{4}}x\right),$ we get
$$
\|E_t\|_{r}=t^{-\frac{N}{4}(1-\frac{1}{r})}\|E_1\|_{r}=t^{-\frac{N}{4}({1\over p}-{{1\over q}})}\|E_1\|_{r}.
$$
Using Proposition \ref{abir2}, we obtain
$$
   \|E_1\|_{r}\leq \mathcal{K} \|F\|_{r}.
$$
By interpolation and since  $ \|F\|_{\infty}\leq \omega,\;  \|F\|_{1}\leq 1,$ we obtain
$$\|F\|_{r}\leq \omega^{1-{1\over r}}\leq 1+\omega.$$
Finally,
$$\|{\rm e}^{-t\Delta^2}\varphi\|_{q}\leq \mathcal{K}(1+\omega)t^{-\frac{N}{4}({1\over p}-{{1\over q}})}\|\varphi\|_{p}.$$
This finishes the proof of Proposition \ref{LPLQQ} with $\mathcal{H}=\mathcal{K}(1+\omega).$
\end{proof}

The following Proposition is a generalization of \cite[Lemma 2.2, {p.} 1176]{Ioku} to the biharmonic operator.

\begin{proposition}\label{pre} Let $1\leqslant p\leqslant 2,\,\,1\leqslant q\leqslant\infty.$ Then the following estimates hold:
\begin{enumerate}
  \item [(i)] $\|{\rm e}^{-t\Delta^2}\varphi\|_{\EP}\leqslant   \mathcal{H}\|\varphi\|_{\EP},\; \forall\; t>0,$\; $\forall\; \varphi\in\EP.$
  \item [(ii)] $\|{\rm e}^{-t\Delta^2}\varphi\|_{\EP}\leqslant  \mathcal{H}\,t^{\frac{-N}{4p}}\left(\log(t^{\frac{-N}{4}}+1)\right)
    ^{-\frac{1}{2}} \|\varphi\|_{p},\; \forall\; t>0,$\; $\forall\; \varphi\in L^p.$
  \item [(iii)] $ \|{\rm e}^{-t\Delta^2}\varphi\|_{\EP}\leqslant  {\mathcal{H}\over \sqrt{\log 2}}\,\left[\,t^{\frac{-N}{4q}}\|\varphi\|_{q}+\|\varphi\|_{2}\right], \; \forall\; t>0,$ $\forall\; \varphi\in L^q\cap L^2.$
\end{enumerate}
\end{proposition}
\begin{proof} We begin by proving (i). For any $\alpha>0,$  expanding the exponential function leads to
\begin{eqnarray*}
 \int_{\R^N}\bigg(\exp\left(\frac{{\rm e}^{-t\Delta^2}\varphi}{\alpha}\right)^2-1\bigg)dx &=& \sum_{k=1}^{\infty}
\frac{\|{\rm e}^{-t\Delta^2}\varphi\|^{2k}_{{2k}}}{k! \alpha^{2k}}.
\end{eqnarray*}
Then by the $L^{2k}-L^{2k}$ estimate of the biharmonic semi-group \eqref{x}, we obtain
\begin{eqnarray*}
\int_{\R^N}\bigg(\exp\left(\frac{{\rm e}^{-t\Delta^2}\varphi}{\alpha}\right)^2-1\bigg)dx
&\leqslant & \sum_{k=1}^{\infty}\frac{\mathcal{H}^{2k}\|\varphi\|^{2k}_{{2k}}}{k! \alpha^{2k}}\\
&=&\int_{\R^N}\left(\exp\left(\frac{\mathcal{H}\varphi}{\alpha}\right)^2-1\right)dx.
\end{eqnarray*}
Therefore we obtain
\begin{eqnarray*}
  \|{\rm e}^{-t\Delta^2}\varphi\|_{\EP} &=& \inf\left\{\alpha>0,\,\,\,\int_{\R^N}\left(\exp
  \left(\frac{{\rm e}^{-t\Delta^2}\varphi}{\alpha}\right)^2-1\right)dx \leq1\right\} \\
  &\leqslant &\inf\left\{\alpha>0,\,\,\,\int_{\R^N}\left(\exp\left(\frac{\mathcal{H}\varphi}{\alpha}\right)^2-1\right)dx \leq1\right\}\\
  &{=}& {\mathcal{H}\|\varphi\|_{\EP}.}
\end{eqnarray*}
This proves (i).

We now turn to the proof of (ii). Using \eqref{x} with $p\leq 2,$ we have
 \begin{eqnarray*}
 \int_{\R^N}\left(\exp\left(\frac{{\rm e}^{-t\Delta^2}\varphi}{\alpha}\right)^2-1\right)dx
 &=&\sum_{k=1}^{\infty}\frac{\|{\rm e}^{-t\Delta^2}\varphi\|^{2k}_{{2k}}}{k!\alpha^{2k}}\nonumber\\
 &\leqslant & \sum_{k=1}^{\infty}\frac{\mathcal{H}^{2k}t^{-\frac{N}{4}(\frac{1}{p}-\frac{1}{2k})2k}\|\varphi\|^{2k}_{p}}{k!\alpha^{2k}}\nonumber\\
  &=&t^{\frac{N}{4}}\left(\exp\left(\frac{\mathcal{H}\, t^{-{\frac{N}{4p}}}\|\varphi\|_{p}}{\alpha}\right)^2-1\right).
\end{eqnarray*}
If we have
$$t^{\frac{N}{4}}\left(\exp\left(\frac{\mathcal{H}\, t^{-{\frac{N}{4p}}}\|\varphi\|_{p}}{\alpha}\right)^2-1\right)\leq 1,$$
then,
$$\mathcal{H}\,t^{-\frac{N}{4p}}\left(\log(t^{-\frac{N}{4}}+1)\right)^{-\frac{1}{2}} \|\varphi\|_{p}\leq \alpha.$$
It follows that
\begin{eqnarray*}
  \|{\rm e}^{-t\Delta^2}\varphi\|_{\EP} &\leqslant & \inf\left\{\alpha>0,\,\,\,\int_{\R^N}
  \left(\exp\left(\frac{{\rm e}^{-t\Delta^2}\varphi}{\alpha}\right)^2-1\right)dx \leq1\right\} \\
  &\leq&\mathcal{H}\,t^{-\frac{N}{4p}}\left(\log(t^{-\frac{N}{4}}+1)\right)^{-\frac{1}{2}} \|\varphi\|_{p}.
\end{eqnarray*}
This proves (ii).

We now prove (iii). By the embedding $L^2\cap L^{\infty}\hookrightarrow
\EP$, we have
\begin{eqnarray*}
   \|{\rm e}^{-t\Delta^2}\varphi\|_{\EP}&\leqslant & {1\over \sqrt{\log2}}\left[ \|{\rm e}^{-t\Delta^2}\varphi\|_{{\infty}}+
   \|{\rm e}^{-t\Delta^2}\varphi\|_{2}\right].
\end{eqnarray*}
Using the  $L^q-L^{\infty}$ estimate \eqref{x}, we get
\begin{eqnarray*}
  \|{\rm e}^{-t\Delta^2}\varphi\|_{\EP} &\leqslant & {\mathcal{H}\over \sqrt{\log2}}\left[t^{-\frac{N}{4q}} \|\varphi\|_{q}+\|\varphi\|_{2}\right].
\end{eqnarray*}
This proves (iii). The proof of the proposition is now complete.
\end{proof}
As a consequence we have the following.
\begin{corollary}
\label{lemme estimates}
Let $N\geq 9, \; q> {N\over 4}.$ Then, for every $g\in L^1\cap L^q,$ we have
 \begin{equation*}
    \|{\rm e}^{-t\Delta^2}g\|_{\EP}\leq \kappa(t)\,\|g\|_{L^1\cap L^q},\; \forall \; t>0,
\end{equation*}
where $\kappa\in L^1(0,\infty)$ is given by
$$\kappa(t)=\frac{ 2\mathcal{H}}{\sqrt{\log 2}}\min\biggr\{ t^{-\frac{N}{4q}}+1,\; t^{-\frac{N}{4}}\Big(\log(t^{-\frac{N}{4}}+1)\Big)^{-\frac{1}{2}}\biggl\}.$$
\end{corollary}
Here we use $\|g\|_{L^1\cap L^q}=\|g\|_{1}+\|g\|_{q}.$
\begin{proof}
We have, by Proposition \ref{pre} (ii),
\begin{equation}\label{ii9}
    \|{\rm e}^{-t\Delta^2}g\|_{\EP}\leq \mathcal{H}\, t^{-\frac{N}{4}}\Big(\log(t^{-\frac{N}{4}}+1)\Big)^{-\frac{1}{2}}\,\|g\|_{1}.
\end{equation}
Using Proposition \ref{pre} (iii) and interpolation inequality, we get
\begin{equation}\label{iii99}
    \|{\rm e}^{-t\Delta^2}g\|_{\EP}\leq \frac{2\mathcal{H}}{\sqrt{\log 2}}\,( t^{-\frac{N}{4q}}+1 )\Big[\|g\|_{q}+\|g\|_{1}\Big].
\end{equation}
Combining the inequalities \eqref{ii9} and \eqref{iii99}, we obtain
\begin{equation}\label{2.8}
    \|{\rm e}^{-t\Delta^2}g\|_{\EP}\leq \kappa(t)\;\left(\|g\|_{1}+\|g\|_{q}\right).
\end{equation}
By the assumption $N\geq9,\,\,q>\frac{N}{4}$, we can see that $\kappa \in L^{1}(0,\,\infty).$
\end{proof}

For $N=8$ we have a similar result. For this we need to introduce an appropriate Orlicz space. Let $\phi(u):={\rm e}^{u^2}-1-u^2$ and $L^{\phi}(\R^8)$ be the associated Orlicz space with the Luxembourg norm \eqref{Luxemb}. From the definition, we have
\begin{equation}\label{slim}
    C_1\|u\|_{\exp L^2(\R^8)}\leqslant\|u\|_{2}+\|u\|_{L^{\phi}(\R^8)}\leqslant C_2\|u\|_{\exp L^2(\R^8)},
\end{equation}
for some $C_1,\,C_2>0.$
\begin{corollary}
\label{forn=8}
For every $g\in L^1(\R^8)\cap L^4(\R^8),$ we have
 \begin{equation}\label{2.8b}
    \|{\rm e}^{-t\Delta^2}g\|_{L^\phi(\R^8)}\leq \zeta(t)\|g\|_{L^1\cap L^4(\R^8)},\; \forall \; t>0,
\end{equation}
where $\zeta\in L^1(0,\infty)$ is given by
$$
\zeta(t)=\frac{\mathcal{H}}{\sqrt{\log 2}}\,\min\biggr\{1+t^{-\frac{1}{2}},\; t^{-2}\Big(\log(t^{-2}+1)\Big)^{-\frac{1}{4}}\biggl\}.
$$
\end{corollary}

\begin{proof} We have, using Proposition \ref{LPLQQ},
\begin{eqnarray*}
  \int_{\R^8}\phi\left(\frac{\left|{\rm e}^{-t\Delta^2}g\right|}{\alpha}\right)\,dx &=&
\sum_{k\geq2}\frac{\|{\rm e}^{-t\Delta^2}g\|^{2k}_{{2k}}}{\alpha^{2k}k!} \\
   &\leq& \sum_{k\geq2}\frac{\mathcal{H}^{2k}t^{-2(1-\frac{1}{2k}){2k}}\|g\|_{1}^{2k}}{\alpha^{2k}k!}\\
   &=& t^{2}\phi\left(\frac{\mathcal{H}t^{-2}\|g\|_{1}}{\alpha}\right) \\
   &\leq& t^{2}\left(\exp\Big\{\left(\frac{\mathcal{H}t^{-2}\|g\|_{1}}{\alpha}\right)^4\Big\}-1\right).
\end{eqnarray*}
Therefore we obtain that
\begin{eqnarray}\label{2.13}
  \|{\rm e}^{-t\Delta^2}g\|_{L^{\phi}(\R^8)} &\leq& \inf\left\{\alpha>0,
  t^{2}\left(\exp\Big\{\left(\frac{\mathcal{H}t^{-2}\|g\|_{1}}{\alpha}\right)^4\Big\}-1\right)\leq1 \right\} \nonumber\\
   &=& \mathcal{H}\,t^{-2}\Big(\log\left(t^{-2}+1\right)\Big)^{-\frac{1}{4}} \|g\|_{1}.
\end{eqnarray}
On the other hand, from the embedding $L^4(\R^8)\cap L^{\infty}(\R^8)\hookrightarrow L^{\phi}(\R^8)$,  we see that
\begin{equation*}
\begin{split}
\|{\rm e}^{-t\Delta^2}g\|_{L^{\phi}(\R^8)}&\leq\,\frac{1}{\sqrt{\log 2}}\Big[\|{\rm e}^{-t\Delta^2}g\|_{{\infty}}+\|{\rm e}^{-t\Delta^2}g\|_{{4}}\Big].
\end{split}
\end{equation*}
Using Proposition \ref{LPLQQ}, we obtain that
\begin{equation}\label{2.14}
\begin{split}
\|{\rm e}^{-t\Delta^2}g\|_{L^{\phi}(\R^8)}&\leq\, \frac{\mathcal{H}}{\sqrt{\log 2}}\, \Big[ t^{-\frac{1}{2}}\|g\|_{{4}}+ \|g\|_{{4}}\Big].
\end{split}
\end{equation}
Combining the inequalities \eqref{2.13} and \eqref{2.14}, we have
\begin{equation*}
   \|{\rm e}^{-t\Delta^2}g\|_{L^{\phi}(\R^8)}\leq \zeta(t)\|g\|_{L^1\cap L^4(\R^8)}.
\end{equation*}
We remark that $\zeta \in L^{1}(0,\infty)$.
\end{proof}
We will also need the following result for the proofs.
\begin{proposition}\label{continu0}
 If $u_0\in \EPO$ then  ${\rm e}^{-t\Delta^2}u_0 \in C([0,\infty); \EPO).$
\end{proposition}
\begin{proof} Let $u_0\in \EPO.$ Similarly as in the proof of Part (i) of Proposition \ref{pre} and by the definition of $\EPO$ in \eqref{exp L0}, we have that ${\rm e}^{-t\Delta^2}u_0\in \EPO$ for every $t>0.$ Thus, it remains only to prove the continuity at $t=0.$ That is
$$\displaystyle\lim_{t\rightarrow0}\|{\rm e}^{-t\Delta^2}u_0-u_0\|_{\EP}=0.$$

Since $u_0\in \EPO,$ there exist $\{u_n\}_{n=1}^{\infty}\subset C_0^{\infty}(\R^N)$  such that $\displaystyle\lim_{n\rightarrow \infty}\|u_n-u_0\|_{\EP}=0.$ Thus, we have
\begin{eqnarray*}
  \|{\rm e}^{-t\Delta^2}u_0-u_0\|_{\EP}&\leq& \|{\rm e}^{-t\Delta^2}(u_0-u_n)\|_{\EP}+\|{\rm e}^{-t\Delta^2}u_n-u_n\|_{\EP} \\
  &&\qquad+  \|u_n-u_0\|_{\EP}.
\end{eqnarray*}
We use the fact that $L^2\cap L^{\infty}\subset\EP$ and Part (i) of Proposition \ref{pre}, we obtain
\begin{eqnarray*}
  \|{\rm e}^{-t\Delta^2}u_0-u_0\|_{\EP}
 &\leq& \frac{1}{\sqrt{\log 2}}\left(\|{\rm e}^{-t\Delta^2}u_n-u_n\|_{2}+\|{\rm e}^{-t\Delta^2}u_n-u_n\|_{{\infty}}\right)\\
&&\qquad+(\mathcal{H}+1)\|u_n-u_0\|_{\EP}.
\end{eqnarray*}
Since $u_n\in C_0^{\infty}(\R^N),$ we have  $\displaystyle\lim_{t\rightarrow0}\left(\|{\rm e}^{-t\Delta^2}u_n-u_n\|_{2}+\|{\rm e}^{-t\Delta^2}u_n-u_n\|_{{\infty}}\right)=0.$
Hence $$\displaystyle\limsup_{t\rightarrow0} \|{\rm e}^{-t\Delta^2}u_0-u_0\|_{\EP}\leq { (\mathcal{H}+1)}\|u_n-u_0\|_{\EP},$$
for every $n\in \mathbb{N}.$ This finishes the proof of the proposition.
\end{proof}

It is known that ${\rm e}^{-t\Delta^2}$ is a $C^0-$semigroup on $L^p.$ By Proposition \ref{continu0}, it is also a $C^0-$semigroup on $ \EPO.$ We will show that is not the case on $\EP.$ That is, we will prove that ${\rm e}^{-t\Delta^2}$ is not a $C^0-$ semigroup on $\EP.$ In fact, we show that   ${\rm e}^{-t\Delta^2}$ is not continuous at $t=0$ in $\EP$. We have the following result.

\begin{proposition}
\label{pdiscontinuite}
There exist $u_0\in \EP$ and a constant $C>0$ such that
\begin{equation}
\label{ineq1}
 \|{\rm e}^{-t\Delta^2} u_0-u_0\|_{\EP} \geqslant C,\; \forall\; t>0.
\end{equation}
\end{proposition}

To prove the previous proposition, we recall the notion of rearrangement of functions. Let $u$ be measurable function defined on $\R^N$ which is finite almost everywhere. The distribution function $\mu_u$ of $u$ is given by
$$\mu_u(\tau)=\Big|\left\{x\in \R^N;\, |u(x)|>\tau\right\}\Big|.$$
The decreasing rearrangement of u is the function $u^{\sharp}$ defined on $[0,\infty)$ by
\begin{equation*}
    u^{\sharp}(r):=\inf\biggr\{\tau>0;\,\mu_u(\tau)\leq r\biggl\}.
\end{equation*}
 Let $u^{\sharp \sharp}$ be the average function of $u^{\sharp},$ namely
\begin{equation}\label{mean}
   u^{\sharp \sharp}(r):=\frac{1}{r}\int_0^ru^{\sharp}(\eta)\,d\eta.
\end{equation}

We recall the following lemma.
\begin{lemma}\label{(5.2)}(\cite{Ioku}) Let $u\in\EP$. Then the
following inequality holds
\begin{equation*}\label{83....}
    \sup_{0<r<1}\frac{u^{\sharp\sharp}(r)}{(\log(\frac{e}{r}))^{\frac{1}{2}}}+\|u\|_{2}\leq C\|u\|_{\EP}.
\end{equation*}
\end{lemma}
We also give the following result for $u_0\in \EP.$
\begin{proposition}
\label{plinfty}
For any $t>0$ and any $u_0\in \EP,$ we have
$$ ({\rm e}^{-t\Delta^2}u_0)^{\sharp\sharp}\in L^\infty(0,\infty).$$
\end{proposition}
\begin{proof} Using Proposition \ref{LPLQQ} and Lemma \ref{sarah55}, we have
$$
\|{\rm e}^{-t\Delta^2}u_0\|_{\infty}\leq \mathcal{H}  t^{-\frac{N}{8}}\|u_0\|_{2}\leq \mathcal{H} t^{-\frac{N}{8}}\|u_0\|_{\EP}<\infty.
$$
Hence ${\rm e}^{-t\Delta^2}u_0\in L^\infty.$ By using the rearrangement property, we have
$$\|{\rm e}^{-t\Delta^2}u_0\|_{{\infty}}=\|({\rm e}^{-t\Delta^2}u_0)^{\sharp\sharp}\|_{L^{\infty}(0,\infty)},$$ and then  $({\rm e}^{-t\Delta^2}u_0)^{\sharp\sharp}\in L^\infty(0,\infty).$
\end{proof}

We now turn to the proof of  inequality \eqref{ineq1}.

\begin{proof}[Proof of  Proposition \ref{pdiscontinuite}] Using Lemma \ref{(5.2)}, we have
\begin{eqnarray*}
  \|{\rm e}^{-t\Delta^2}u_0-u_0\|_{\EP} &\geq& C\sup_{0<r<1}\frac{\left({\rm e}^{-t\Delta^2}u_0-u_0\right)^{\sharp\sharp}(r)}{\left(\log\frac{e}{r}\right)^{\frac{1}{2}}} \\
   &\geq& C \lim_{r\rightarrow 0} \frac{\left({\rm e}^{-t\Delta^2}u_0-u_0\right)^{\sharp\sharp}(r)}{\left(\log\frac{e}{r}\right)^{\frac{1}{2}}}.
\end{eqnarray*}
Thanks to the triangle inequality of the average function: $(u+v)^{\sharp\sharp}(r)\leq u^{\sharp\sharp}(r)+v^{\sharp\sharp}(r),$
 and the mean value \eqref{mean}, we obtain
\begin{equation*}
    \|{\rm e}^{-t\Delta^2}u_0-u_0\|_{\EP}\geq C \lim_{r\rightarrow 0} \frac{u_0^{\sharp\sharp}(r)- ({\rm e}^{-t\Delta^2}u_0)^{\sharp\sharp}(r)}{\left(\log\frac{e}{r}\right)^{\frac{1}{2}}}.
\end{equation*}
By Proposition \ref{plinfty}, we see that
$$\lim_{r\rightarrow 0}\frac{({\rm e}^{-t\Delta^2}u_0)^{\sharp\sharp}(r)}{\left(\log\frac{e}{r}\right)^{\frac{1}{2}}}=0.$$
Thus, we get
\begin{equation}
\label{ineqgeneral}
    \|{\rm e}^{-t\Delta^2}u_0-u_0\|_{\EP}\geq C \lim_{r\rightarrow 0} \frac{u_0^{\sharp\sharp}(r)}{\left(\log\frac{e}{r}\right)^{\frac{1}{2}}}.
\end{equation}
To conclude we make the choice of $u_0$ as follows. Let $\omega_N$ be measure of the unit ball in $\R^N$ and
\begin{equation*}
    u_0(x):=\left\{
              \begin{array}{ll}
              \displaystyle \frac{1-2\log\left(\omega_N|x|^N\right)}{2\sqrt{1-\log(\omega_N|x|^N)}}, & \hbox{if} \,\,0<|x|<1, \\
                0, &  \hbox{if}\; |x|\geq 1.
              \end{array}
            \right.
\end{equation*}
Then we have $$u_0^{\sharp\sharp}(r)=\left(\log\frac{e}{r}\right)^{\frac{1}{2}},\; \mbox{ for }\; 0<r<\omega_N.$$ Therefore, using \eqref{ineqgeneral}, we obtain
\begin{equation*}
    \|{\rm e}^{-t\Delta^2}u_0-u_0\|_{\EP}\geq C .
\end{equation*}
This finishes the proof of Proposition \ref{pdiscontinuite}.
\end{proof}

\section{Local well-posedness in $\EPO$}

In this section we prove the existence and the uniqueness of solution to \eqref{1.1} in $C([0,T]; \EPO)$ for some $T>0$, namely Theorem \ref{local}. Throughout this section we assume that the nonlinearity $f :\R\to\R$ satisfies $f(0)=0$ and
\begin{equation}\label{Nonlin}
|f(u)-f(v)|\leq C|u-v|\left({\rm e}^{\lambda \,u^2}+{\rm e}^{\lambda \,v^2}\right),\quad\forall\;u,\; v\in\R
\end{equation}
for some constants $C>0,\;\lambda>0$. We emphasize that, thanks to Proposition \ref{step1}, the Cauchy problem \eqref{1.1} admits the equivalent integral formulation \eqref{integral}.
\begin{proposition}\label{4.1} Let $T > 0$ and $u_0$ be in $\EPO$. If $u$ belongs to $C([0,T]; \EPO)$, then $u$ is a weak solution of \eqref{1.1} if and only if $u(t)$ satisfies the integral equation \eqref{integral} for any  $t\in(0, T )$.
\end{proposition}

Now we are ready to prove Theorem \ref{local}. As explained in the introduction, the idea here is to split the initial data $u_0\in\EPO$ into a small part in $\EP$ and a smooth one. This will be done using the density of $C^\infty_0(\R^N)$ in $\EPO$. First we solve the initial value problem with smooth initial data to obtain a local and bounded solution $v$. Then we consider the perturbed equation satisfied by $w := u-v$ and with small initial data. Now we come to the details. For $\varepsilon>0$ to be chosen later, we write $u_0=v_0+w_0$, where $v_0\in C_0^\infty(\R^N)$ and $\|w_0\|_{\EP}\leq \varepsilon$. Then, we consider the two Cauchy problems:
\begin{equation*}
(\mathcal{P}_1)\qquad\left\{
\begin{array}{ll}
\partial_t v+\Delta^2 v=f(v), & \qquad t>0,\,x\in\R^N,\qquad\qquad\qquad\\
v(0)=v_0,&
\end{array}
\right.
\end{equation*}

and

\begin{equation*}
(\mathcal{P}_2)\qquad\left\{
\begin{array}{ll}
\partial_t w+\Delta^2 w=f(w+v)-f(v), & \qquad t>0,\,x\in\R^N, \\
w(0)=w_0.&
  \end{array}
\right.
\end{equation*}
We prove the following existence result concerning $(\mathcal{P}_1).$
\begin{proposition}\label{C0} Let $v_0\in L^2\cap L^{\infty}$. Then there exist a time $T>0$ and a mild solution $v\in C([0,T],\EPO)\cap L^{\infty}(0,T;L^{\infty})$ to $(\mathcal{P}_1)$.
\end{proposition}
\begin{proof}[Proof of Proposition \ref{C0}] We use a fixed point argument. We introduce, for any positive time $T$  the following complete metric space
$${\mathcal Y}(T):=\Big\{\;v\in C([0,T];\EPO)\cap L^{\infty}(0,T;L^{\infty});\quad \|v\|_{T}\leq 2{\mathcal H}\|v_0\|_{L^2\cap L^{\infty}}\Big\},$$
where $\|v\|_{T}:=\|v\|_{L^{\infty}(0,T;L^{2})}+\|v\|_{L^{\infty}(0,T;L^{\infty})},$ and $\|v_0\|_{L^2\cap L^{\infty}}=\|v_0\|_{2}+\|v_0\|_{{\infty}}$.
Set
$$\Phi(v)(t):={\rm e}^{-t\Delta^2}v_0+\int_0^{t}{\rm e}^{-(t-s)\Delta^2}f(v(s))\,ds.$$

We will prove that if $T>0$ is small enough then $\Phi$ is a contraction map from ${\mathcal Y}(T)$ into itself. First let us remark that by Proposition \ref{continu0} and the fact that $f(v)\in L^1(0,T; \EPO)$ whenever $v\in C([0,T];\EPO)\cap L^{\infty}(0,T;L^{\infty})$,
we have $\Phi(v)\in C([0,T];\EPO)\cap L^{\infty}(0,T;L^{\infty})$. Now, for every $v_1,\, v_2\in {\mathcal Y}(T)$, we have thanks to \eqref{Nonlin},
\begin{eqnarray*}
\|\Phi(v_1)-\Phi(v_2)\|_{L^{\infty}(0,T;L^q)}&\leq& {\mathcal H}\int_0^T \|f(v_1(s))-f(v_2(s))\|_{q}\,ds\\
 &\leq& {\mathcal H}T \|f(v_1)-f(v_2)\|_{L^{\infty}(0,T;L^q)}\\
&\leq& C{\mathcal H}T\left({\rm e}^{\lambda  \|v_1\|^2_{L^\infty_t(L^\infty_x)}}+{\rm e}^{\lambda\|v_2\|^2_{L^\infty_t(L^\infty_x)}}\right)\|v_1-v_2\|_{L^\infty(0,T; L^q)}
\end{eqnarray*}
where $q=2$ or $q=\infty$. Then, it follows that
\begin{equation}\label{4.3}
 \|\Phi(v_1)-\Phi(v_2)\|_{T}\leq 2C\,{\mathcal H}\,T\,{ {\rm e}^{4\lambda {\mathcal H}^2\|v_0\|_{L^2\cap L^\infty}^2}}\|v_1-v_2\|_{T}.
\end{equation}
Similarly we have
\begin{equation}\label{4.6}
  \|\Phi(v)\|_{T}  \leq C\,{\mathcal H}\,T\,{ {\rm e}^{4\lambda {\mathcal H}^2\|v_0\|_{L^2\cap L^\infty}^2}}\,\|v\|_{T}+\mathcal{H}\|v_0\|_{L^2\cap L^{\infty}}.
\end{equation}
From \eqref{4.3} and \eqref{4.6} we conclude that for $T=T(\|v_0\|_{L^2\cap L^\infty})>0$ small enough, $\Phi$ is a contraction map on ${\mathcal Y}(T)$. This finishes the proof of Proposition \ref{C0}.
\end{proof}

We now prove the following concerning problem $(\mathcal{P}_2).$
\begin{proposition}\label{UU2} Let $T>0$ and $v\in L^{\infty}(0,T;L^{\infty})$. Let $w_0\in\EPO$. Then for $\|w_0\|_{\EP}\leq \varepsilon$, with $\varepsilon>0$ small enough,
there exist a time $\widetilde{T}=\widetilde{T}(w_0,\varepsilon,v)>0$ and a mild solution $w\in C([0,\widetilde{T}],\EPO)$ to problem $(\mathcal{P}_2)$.
\end{proposition}
To prove Proposition \ref{UU2} we establish first the following lemma.
\begin{lemma}\label{4.3333}Let $v\in L^{\infty}$ and $w_1,\,w_2\in \EP$ with $\|w_1\|_{\EP},\, \|w_2\|_{\EP}\leq { K}$ for some constant ${ K}>0$. Let $2\leq q<\infty$, and assume that $4\lambda q { K}^2\leq 1$ where $\lambda$ is given by \eqref{Nonlin}. Then there exists a constant $C_q=C(q)>0$ such that
$$\Big\|f(w_1+v)-f(w_2+v)\Big\|_{q}\leq\,C_q\,{\rm e}^{2\lambda \|v\|_{\infty}^2}\Big\|w_1-w_2\Big\|_{\EP}.$$
\end{lemma}
\begin{proof}[Proof of the Lemma \ref{4.3333}] By the assumption \eqref{Nonlin} on $f$, we have
\begin{eqnarray*}
\Big\|f(w_1+v)-f(w_2+v)\Big\|_{q}&\leq& C \Big\||w_1-w_2|\left( {\rm e}^{2\lambda w_1^2+2\lambda v^2}+{\rm e}^{2\lambda w_2^2+2\lambda v^2}\right)\Big\|_{q}\\&&\hspace{-4cm}\leq  {\rm e}^{2\lambda \|v\|_{\infty}^2}\bigg(2C\Big\|w_1-w_2\Big\|_{q}+C\Big\||w_1-w_2|\left({\rm e}^{2\lambda w_1^2}-1\right)\Big\|_{q}\bigg)\\&&\hspace{-3cm}+C{\rm e}^{2\lambda \|v\|_{\infty}^2}\,\Big\||w_1-w_2|\left({\rm e}^{2\lambda w_2^2}-1\right)\Big\|_{q}\\&&\hspace{-4cm}\leq{\rm e}^{2\lambda \|v\|_{\infty}^2}\bigg(2C\Big\|w_1-w_2\Big\|_{q}+C\Big\|w_1-w_2\Big\|_{{2q}}\Big\|{\rm e}^{2\lambda w_1^2}-1\Big\|_{{2q}}\bigg)\\&&\hspace{-3cm}+C{\rm e}^{2\lambda \|v\|_{\infty}^2}\,\Big\|w_1-w_2\Big\|_{{2q}}\Big\|{\rm e}^{2\lambda w_2^2}-1\Big\|_{{2q}}\\&&\hspace{-4cm}\leq\,C_q\,{\rm e}^{2\lambda \|v\|_{\infty}^2}\Big\|w_1-w_2\Big\|_{\EP},
\end{eqnarray*}
where we have used H\"older inequality, Lemma \ref{sarah55} and Lemma \ref{med}. This finishes the proof of Lemma \ref{4.3333}.
\end{proof}
Now we come to the proof of Proposition \ref{UU2}.
\begin{proof}[Proof of Proposition \ref{UU2}]
For $\widetilde{T}>0$, $\widetilde{M}>0$, we define
$$W_{\widetilde{M},\widetilde{T}}:=\left\{w\in C([0,\widetilde{T}],\EPO);\; \|w\|_{L^{\infty}(0,\widetilde{T};\EP)}\leq \widetilde{M}\right\}.$$ On the space $W_{\widetilde{M},\widetilde{T}}$ consider the map
$$\mathcal{F}(w)(t):={\rm e}^{-t\Delta^2}w_0+\int_0^{t}{\rm e}^{-(t-s)\Delta^2}\Big(f(w(s)+v(s))-f(v(s))\Big)\,ds.$$
We will prove that for $\widetilde{M}>0$ and $\widetilde{T}>0$ sufficiently small,  $\mathcal{F}$ is a contraction map from $W_{\widetilde{M},\widetilde{T}}$ into itself. From the embedding $L^2\cap L^{\infty}\hookrightarrow\EP$, we have
\begin{equation}
\label{(4.11)}
 \|\mathcal{F}(w_1)-\mathcal{F}(w_2)\|_{\EP}\leq\frac{1}{\sqrt{\log 2}}\,\bigg(\|\mathcal{F}(w_1)-\mathcal{F}(w_2)\|_{2}+\|\mathcal{F}(w_1)-\mathcal{F}(w_2)\|_{{\infty}}\bigg).
\end{equation}
Fix $p>\max\left(2,\,\frac{N}{4}\right).$ Then, thanks to Proposition \ref{LPLQQ} and Lemma \ref{4.3333}, we have
\begin{eqnarray}\label{(4.12)}\nonumber
\|\mathcal{F}(w_1)-\mathcal{F}(w_2)\|_{{\infty}}&\leq &\mathcal{H}\int_0^{t}(t-s)^{-\frac{N}{4p}}\|f(w_1(s)+v(s))-f(w_2(s)+v(s))\|_{p}\,ds\\
  &\leq&C_p\,\mathcal{H}\,{\rm e}^{2\lambda \|v\|_{\infty}^2}\biggr(\int_0^t\,(t-s) ^{-\frac{N}{4p}}\,ds\biggl)\|w_1-w_2\|_{\EP}\nonumber\\
  &\leq&  C_p\,\mathcal{H}\, {\rm e}^{2\lambda \|v\|_{\infty}^2}\widetilde{T}^{1-\frac{N}{4p}}\|w_1-w_2\|_{\EP}.
\end{eqnarray}
Applying again Lemma \ref{4.3333} with $q=2$ and $\widetilde{M}$ sufficiently small, it holds
\begin{eqnarray}\label{(4.13)}
  \|\mathcal{F}(w_1)-\mathcal{F}(w_2)\|_{2} &\leq&C \mathcal{H}\, {\rm e}^{2\lambda \|v\|_{\infty}^2}\int_0^t\|w_1-w_2\|_{\EP} ds\nonumber\\
  &\leq&  C \mathcal{H}\, {\rm e}^{2\lambda \|v\|_{\infty}^2}\widetilde{T}\|w_1-w_2\|_{\EP}.
\end{eqnarray}
Plugging \eqref{(4.11)}, \eqref{(4.12)} and \eqref{(4.13)} together, we get, for $w_1,w_2\in W_{\widetilde{M}, \widetilde{T}}$,
\begin{equation}\label{(4.14)}
\|\mathcal{F}(w_1)-\mathcal{F}(w_2)\|_{L^{\infty}(0, \widetilde{T};\,\EP)}\leq C{\rm e}^{2\lambda \|v\|_{\infty}^2}\left(\widetilde{T}+\widetilde{T}^{1-\frac{N}{4\,p}}\right)\|w_1-w_2\|_{L^{\infty}{(0, \widetilde{T}};\,\EP)}.
\end{equation}
The estimates \eqref{(4.12)} and \eqref{(4.13)} with $w_2=0$ show that the nonlinear term satisfies
$$\mathcal{F}(w)-{\rm e}^{-t\Delta^2}w_0\in L^{\infty}(0,\widetilde{T};\, L^2\cap L^{\infty}).$$
Since $f(w+v)-f(v) \in L^1(0, \widetilde{T}; \EPO)$, it follows that
\begin{equation}\label{(4.15)}
 \mathcal{F}(w)-{\rm e}^{-t\Delta^2}w_0\in C([0,\widetilde{T}];\exp L_0^2).
\end{equation}
Moreover, by Proposition \ref{continu0}, we have
$${\rm e}^{-t\Delta^2}w_0\in C([0,\widetilde{T}];\exp L^2_0).$$
This and \eqref{(4.15)} proves that $ \mathcal{F}(w)\in  C([0,\widetilde{T}];\exp L^2_0).$ Now again by \eqref{(4.14)} with $w_2=0,$ we have
\begin{equation}\label{(4.16)}
\|\mathcal{F}(w)\|_{L^{\infty}(\exp L^2)}\leq \mathcal{H}\|w_0\|_{\exp L^2}+ C{\rm e}^{2\lambda \|v\|_{\infty}^2}\left(\widetilde{T}+\widetilde{T}^{1-\frac{N}{4\,p}}\right)\|w\|_{L^{\infty}(\exp L^2)}.
\end{equation}
 Choose $\varepsilon>0$, $\widetilde{M}>0$ and $\widetilde{T}>0$ such that
\begin{equation}
\label{lem4.4hold}
4\lambda p \widetilde{M}^2\leq 1,\; p>\max\left(2, N/4\right),\end{equation}
\begin{equation}
\label{contraction}
C{\rm e}^{2\lambda \|v\|_{\infty}^2}\left(\widetilde{T}+\widetilde{T}^{1-\frac{N}{4\,p}}\right)<1,
\end{equation}
and
\begin{equation}
\label{stability}
\mathcal{H}\varepsilon + C{\rm e}^{2\lambda \|v\|_{\infty}^2}\left(\widetilde{T}+\widetilde{T}^{1-\frac{N}{4\,p}}\right)\widetilde{M}\leq \widetilde{M}.
\end{equation}
In particular, Lemma \ref{4.3333} holds for  $K=\mathcal{H}\varepsilon$ or $K=\widetilde{M}$ and $q=p$.
Hence, $\mathcal{F}$ is a contraction map from $W_{\widetilde{M}, \widetilde{T}}$ into itself. This proves Proposition \ref{UU2}.
\end{proof}
We now prove that  if $v$ and $w$ are mild solutions of $(\mathcal{P}_1)$ and $(\mathcal{P}_2)$ respectively, then $u:=v+w$ is a solution of \eqref{integral}.
\begin{proof}[Proof of the existence part in Theorem \ref{local}]  Let $p>\max(2,{N\over 4}).$ Fix $\varepsilon>0$ such that  $$\left(4\lambda p\right)^{1/2}2\mathcal{H}\varepsilon\leq {1\over 2}.$$ Next one can decompose $u_0 = v_0 + w_0$ with $v_0\in C_0^{\infty}(\R^N)$ and $\|w_0\|_{\EP} < \varepsilon$.
By Proposition \ref{C0}, there exists $v$ solution of $(\mathcal{P}_1)$ on $[0,T_1], T_1=T_1(\|v_0\|_{L^2\cap L^{\infty}})$ and $\|v\|_{T_1}\leq 2{\mathcal H}\|v_0\|_{L^2\cap L^{\infty}}.$ Choose $\widetilde{T}>0,\; \widetilde{T}<T_1, $ such that
$$C{\rm e}^{8\lambda {\mathcal H}^2\|v_0\|_{L^2\cap L^{\infty}}^2}\left(\widetilde{T}+\widetilde{T}^{1-\frac{N}{4\,p}}\right)<{1\over 2}.$$ Choose now $\widetilde{M}>0$ such that
$$\mathcal{H}\varepsilon+{1\over 2}\widetilde{M}\leq \widetilde{M},\; \widetilde{M} \leq {1\over \left(4\lambda p\right)^{1/2}},$$
that is $$2\mathcal{H}\varepsilon\leq \widetilde{M}\leq {1\over \left(4\lambda p\right)^{1/2}}.$$
This is possible by the condition on $\varepsilon$.

Now the assumptions on $\widetilde{T}$ and $\widetilde{M}$ imply that
$$\mathcal{H}\varepsilon+ C{\rm e}^{8\lambda {\mathcal H}\|v_0\|_{L^2\cap L^{\infty}}^2}\left(\widetilde{T}+\widetilde{T}^{1-\frac{N}{4\,p}}\right)\widetilde{M}\leq \mathcal{H}\varepsilon+{1\over 2}\widetilde{M}\leq \widetilde{M},$$
$$C{\rm e}^{8\lambda {\mathcal H}^2\|v_0\|_{L^2\cap L^{\infty}}^2}\left(\widetilde{T}+\widetilde{T}^{1-\frac{N}{4\,p}}\right)<1.$$
and $$4\lambda p\widetilde{M}^2\leq 1.$$
Then $\widetilde{M}$ satisfies the hypotheses of  Lemma \ref{4.3333} with $K=\widetilde{M}.$ Using the fact that $\|v\|_{T_1}\leq 2{\mathcal H}\|v_0\|_{L^2\cap L^{\infty}},\; \widetilde{T}<T_1$, then \eqref{(4.16)} and \eqref{(4.14)} are satisfied. Then by Proposition \ref{UU2} there exists $w$ a solution of $(\mathcal{P}_2)$ on $[0, \widetilde{T}].$ Hence $u:=v+w$ is a solution of \eqref{integral} in $C([0,\min\{T_1,\widetilde{T}\}=T]; \EPO)$.
\end{proof}

Let us now prove that $u$ is solution of \eqref{1.1}.

\begin{proof}[Proof of the equivalence between the differential and the integral equations] We claim that the solution $u$ constructed above belongs to $L^{\infty}_{loc}(0,T;L^{\infty})$. Let $u\in C([0, T];\EPO)$ be the solution of the integral equation \eqref{integral}. Then using smoothing effect \eqref{x}, Lemma \ref{sarah55} and Proposition \ref{pre}, we get for $t>0$,
\begin{eqnarray*}
\|{\rm e}^{-t\Delta^2}u_0\|_{{\infty}}&\leq&\mathcal{H}\left({t\over 2}\right)^{-{N\over 8}}\, \|{\rm e}^{-{t\over 2}\Delta^2}u_0\|_{2}\\
&\leq& C t^{-{N\over 8}}\,\|{\rm e}^{-{t\over 2}\Delta^2}u_0\|_{\EP}\\
&\leq&C t^{-{N\over 8}}\,\|u_0\|_{\EP}<\infty,
 \end{eqnarray*}
 and hence  ${\rm e}^{-t\Delta^2}u_0\in L^{\infty}$. Let us consider now the nonlinear term. Fix $p > \max\left(2,\,\frac{N}{4}\right)$. It follows from \eqref{x} and the assumption \eqref{Nonlin} on $f$  that there exists some  constant $\tilde{C}>\lambda \,p$, such that for any $0 < t < T$,
\begin{eqnarray*}
\int_0^t\|{\rm e}^{-(t-s)\Delta^2}f(u(s))\|_{{\infty}}\, ds&\leq& \mathcal{H}\int_0^t (t-s)^{-\frac{N}{4 p}}\|f(u(s))\|_{p} \,ds\\
&\leq& C\int_0^t (t-s)^{-\frac{N}{4 p}}\left(\int_{\R^N}|u(s)|^p\,{\rm e}^{\lambda pu^2(s)}\,dx\right)^{\frac{1}{p}}\,ds\\
&\leq&C\int_0^t (t-s)^{-\frac{N}{4 p}}\left(\int_{\R^N}\left({\rm e}^{\tilde{C}u^2(s)}-1\right)\,dx\right)^{\frac{1}{p}}\,ds\\
&\leq&C t^{1-\frac{N}{4 p}}\displaystyle\sup_{0\leq s\leq T}\left(\int_{\R^N}\left({\rm e}^{\tilde{C}u^2(s)}-1\right)\,dx\right)^{\frac{1}{p}}.
\end{eqnarray*}
Since $u\in C([0, T];\EPO)$, then by Proposition \ref{step1}
\begin{equation*}
   \sup_{0\leq s\leq T}\left(\int_{\R^N}\left({\rm e}^{\tilde{C}u^2(s)}-1\right)\,dx\right)^{\frac{1}{p}}< \infty.
\end{equation*}
This shows that $u\in L^{\infty}_{loc}(0,T;L^{\infty})$ and the claim follows.
\end{proof}
\begin{proof}[Proof of the uniqueness part in Theorem \ref{local}] Let $u,\, v\in C([0,T];\EPO)$ be two solutions of \eqref{integral} with the same initial data $u(0)=v(0)=u_0.$ Let
$$t_0=\sup\biggl\{t\in[0,T]\,\;\;\mbox{such \,that}\,\,\;\;u(s)=v(s)\,\,\;\; \mbox{for  every}\,\;\;s\in[0,t] \,\biggr\}.$$
Let us suppose by contradiction that $0\leq t_0<T.$ Since $u(t)$ and $v(t)$ are continuous in time we have $u(t_0)=v(t_0).$ Let us denote  by $\tilde{u}(t):=u(t+t_0)$ and $\tilde{v}(t):=v(t+t_0)$. Then $\tilde{u}$ and $\tilde{v}$ satisfy \eqref{integral} on $(0,T-t_0]$ and $\tilde{u}(0)=\tilde{v}(0)=u(t_0).$ We will prove that there exists a positive time $0<\tilde{t}\leq T-t_0$ such that
\begin{equation}\label{5.1}
\sup_{0<t<\tilde{t}}\|\tilde{u}-\tilde{v}\|_{\EP}\leq C(\tilde{t})\, \sup_{0<t<\tilde{t}}\|\tilde{u}-\tilde{v}\|_{\EP},
\end{equation}
for a constant $C(\tilde{t})<1,$ and so $\tilde{u}(t)=\tilde{v}(t)$ for any $t\in[0,\tilde{t}].$ Therefore $u(t+t_0)=v(t+t_0)$ for any $t\in[0,\tilde{t}]$ in contradiction with the definition of $t_0$. In order to establish inequality \eqref{5.1} we control both the $L^2-$norm and the $L^{\infty}-$norm of the
difference of the two solutions. Thanks to Proposition \ref{LPLQQ} and H\"older inequality for some $p,q$ such that $\frac{1}{p}+\frac{1}{q}=\frac{1}{2}$ and $p, q>2,$ we have
\begin{eqnarray*}
&&\|\tilde{u}(t)-\tilde{v}(t)\|_{2}\\
&&\qquad\leq\mathcal{H}C\int_0^t\bigg\||\tilde{u}(s)-\tilde{v}(s)|({\rm e}^{\lambda\tilde{u}^2(s)}+{\rm e}^{\lambda\tilde{v}^2(s)})\bigg\|_{2}\,ds \\
&&\qquad\leq C\int_0^t \|\tilde{u}(t)-\tilde{v}(t)\|_{2}ds\\
&&\qquad\qquad+C\int_0^t\left(\|\tilde{u}(s)-\tilde{v}(s)\|_{q}\bigg\|({\rm e}^{\lambda\tilde{u}^2(s)}-1)+
({\rm e}^{\lambda\tilde{v}^2(s)}-1)\bigg\|_{p}\right)\,ds\\
&&\qquad\leq C t\sup_{0<s<t}\|\tilde{u}(t)-\tilde{v}(t)\|_{\EP} \\
&&\qquad \qquad+C t\sup_{0<s<t}\|\tilde{u}(t)-\tilde{v}(t)\|_{\EP}\int_0^t\bigg\|\left({\rm e}^{\lambda\tilde{u}^2(s)}-1\right)+\left({\rm e}^{\lambda\tilde{v}^2(s)}-1\right)\bigg\|_{p}\,ds.
\end{eqnarray*}
Moreover, thanks to Proposition \ref{step1}, the term in the integral is uniformly bounded in time. Indeed,
\begin{eqnarray}\label{5.2}
&& \displaystyle\sup_{0<s<T-t_0} \bigg\|\left({\rm e}^{\lambda\tilde{u}^2(s)}-1\right)+\left({\rm e}^{\lambda\tilde{v}^2(s)}-1\right)\bigg\|_{p}
\qquad\qquad\qquad\qquad\qquad\qquad\qquad\\ &&\nonumber\qquad\qquad\qquad\leq\displaystyle\sup_{0<s<T-t_0}\biggl[\bigg(\int_{\R^N}({\rm e}^{\lambda p\tilde{u}^2(s)}-1)ds\bigg)^{1/p}+\bigg(\int_{\R^N}({\rm e}^{\lambda p\tilde{v}^2(s)}-1)ds\bigg)^{1/p}\biggr]\\
&&\qquad\qquad\qquad\leq C(T,t_0,\tilde{u},\tilde{v})<\infty.\nonumber
\end{eqnarray}
Thus, we obtain
\begin{equation}\label{5.3}
\displaystyle\sup_{0<s<t}\|\tilde{u}(t)-\tilde{v}(t)\|_{2} \leq C(T,t_0,\tilde{u},\tilde{v}) t\displaystyle\sup_{0<s<t}\|\tilde{u}(t)-\tilde{v}(t)\|_{\EP}.
\end{equation}
Similarly, we have
\begin{eqnarray*}
&&\!\!\!\!\|\tilde{u}(t)-\tilde{v}(t)\|_{{\infty}}\\
&&\leq C\mathcal{H}\int_0^t(t-s)^{-\frac{N}{4p}}\bigg\||\tilde{u}(s)-\tilde{v}(s)|({\rm e}^{\lambda\tilde{u}^2(s)}+{\rm e}^{\lambda\tilde{v}^2(s)})\bigg\|_{p}\,ds \\
&&\leq C\int_0^t(t-s)^{-\frac{N}{4p}} \|\tilde{u}(t)-\tilde{v}(t)\|_{p}ds\\
&&+C\int_0^t(t-s)^{-\frac{N}{4p}}\left(\|\tilde{u}(s)-\tilde{v}(s)\|_{{\bar{q}}}
\bigg\|({\rm e}^{\lambda\tilde{u}^2(s)}-1)+({\rm e}^{\lambda\tilde{v}^2(s)}-1)
\bigg\|_{{\bar{p}}}\right)\,ds,
\end{eqnarray*}
for some $p>\max\left(2,\,\frac{N}{4}\right)$ and some $\bar{q},\,\bar{p}$ such that $\frac{1}{\bar{q}}+\frac{1}{\bar{p}}=\frac{1}{p}.$ Since $\bar{p}, \bar{q}\geq 2$, one can apply an estimate similar to \eqref{5.2} via Lemma \ref{sarah55} and Proposition \ref{step1}, and obtain that
\begin{equation}\label{5.4}
   \displaystyle\sup_{0<s<t}\|\tilde{u}(t)-\tilde{v}(t)\|_{\infty} \leq C(T,t_0,\tilde{u},\tilde{v}) t^{1-\frac{N}{4p}}\displaystyle\sup_{0<s<t}\|\tilde{u}(t)-\tilde{v}(t)\|_{\EP}.
\end{equation}
Therefore the two inequalities \eqref{5.3} and \eqref{5.4} with the embedding $L^2\cap L^\infty \hookrightarrow \EP$ imply
\begin{equation*}
   \displaystyle\sup_{0<s<t}\|\tilde{u}(t)-\tilde{v}(t)\|_{\EP} \leq C(T,t_0,\tilde{u},\tilde{v}) (t^{1-\frac{N}{4p}}+t)\displaystyle\sup_{0<s<t}\|\tilde{u}(t)-\tilde{v}(t)\|_{\EP},
\end{equation*}
and for $t$ small enough we obtain the desired estimate. This finishes the proof of Theorem \ref{local}.
\end{proof}

\section{Global Existence}

This section is devoted to the proof of Theorem \ref{GE}. The proof uses a fixed point argument on the associated integral equation
 \begin{equation}\label{NN9}
   u(t)= {\rm e}^{-t\Delta^2}u_{0}+\int_{0}^{t}{\rm e}^{-(t-s)\Delta^2}\,f(u(s)) ds,
\end{equation}
 where  $\|u_0\|_{\EP}\leq \varepsilon$, with  small $\varepsilon>0$ to be fixed later. The  nonlinearity $f$ satisfies $f(0)=0$ and
\begin{equation}\label{NNN9}
|f(u)-f(v)|\leq C \left|u-v\right|\bigg(|u|^{m-1}{\rm e}^{\lambda u^2}+|v|^{m-1}{\rm e}^{\lambda v^2}\bigg),
\end{equation}
for some constants $C>0$ and $\lambda>0,$ and $m$ is a real number larger than $1+{8\over N},$ and $m\geq 2.$ From \eqref{NNN9}, we deduce that
\begin{equation}
\label{taylorm}
 |f(u)-f(v)|\leq C|u-v|\sum_{k=0}^{\infty} \frac{\lambda^{k}}{k!} \left(|u|^{2k+m-1}+|v|^{2k+m-1}\right).
\end{equation}
We will perform a fixed point argument on a suitable metric space.  For $M>0$ we  introduce the space
$$Y_M :=\left\{u\in L^\infty(0,\infty, \EP);\;\displaystyle\sup_{t>0}  t^{\sigma}\|u(t)\|_{p}+\|u\|_{L^{\infty}(0,\infty;\EP)}\leq M\right\},$$ where $p>{N(m-1)\over 4}(\geq 2)$ and $$\sigma={1\over m-1}-\frac{N}{4p}={N\over 4}\left({4\over N(m-1)}-{1\over p}\right)>0.$$ Endowed with the metric $d(u,v)=\displaystyle\sup_{t>0} \Big(t^{\sigma}\|u(t)-v(t)\|_{p}\Big)$, $Y_M$ is a complete metric space. This follows by Proposition \ref{fatou}.  Note that in  our case the parameter $\sigma$ depends on $m$. This is not the case in \cite{Ioku}. A similar choice of $\sigma$ was performed in \cite{CW} for the heat equation with a power nonlinearity.

For $u\in Y_M,$ we define $\Phi(u)$ by
\begin{equation}\label{Phi}
    \Phi(u)(t):={\rm e}^{-t\Delta^2}u_{0}+\int_{0}^{t}{\rm e}^{-(t-s)\Delta^2}(f(u(s))) ds.
\end{equation}
By Proposition \ref{LPLQQ}, Lemma \ref{sarah55} and since $p>{N(m-1)\over 4}\geq 2$, we have
\begin{equation}
\label{med3}
t^{\sigma}\|{\rm e}^{-t\Delta^2}u_{0}\|_{p}\leq  {\mathcal H}\,t^{\sigma-(\frac{1}{m-1}-\frac{N}{4p})}\|u_{0}\|_{{N(m-1)/4}}\leq C\,\|u_{0}\|_{\EP}.
\end{equation}

To show that $\Phi$ is a contraction on $Y_M$, we will treat the cases $N\geq 9$, $N=8$ and $1\leq N\leq 7$ separately.


\subsection{The case $N\geq 9$}



Let $u\in Y_M$. Using Proposition \ref{pre} and Corollary \ref{lemme estimates}, we get for $q>{N\over 4}$,
\begin{eqnarray*}
\|\Phi(u)(t)\|_{\EP}&\leq&\|{\rm e}^{-t\Delta^2}u_{0}\|_{\EP}+\int_{0}^{t}\left\|{\rm e}^{-(t-s)\Delta^2}
(f(u(s))) \right\|_{\EP}\,ds\\
&\leq&\|{\rm e}^{-t\Delta^2}u_{0}\|_{\EP}+\int_{0}^{t}\kappa(t-s)\bigg(\|f(u(s))\|_{L^1\cap L^q}
\bigg)\,ds\\
&\leq&\|{\rm e}^{-t\Delta^2}u_{0}\|_{\EP}+ \|f(u)\|_{L^{\infty}(0,\infty;(L^1\cap L^q))}\int_{0}^{\infty}\kappa(s)\,ds\\
&\leq& \|{\rm e}^{-t\Delta^2}u_{0}\|_{\EP}+C \|f(u)\|_{L^{\infty}(0,\infty;(L^1\cap L^q))}.
\end{eqnarray*}
Hence by Part (i) of Proposition \ref{pre}, we get
$$
\|\Phi(u)\|_{L^\infty(0{,\infty; \EP)}}\leq \mathcal{H} \|u_{0}\|_{\EP}+ C \|f(u)\|_{L^{\infty}(0,\infty;\,L^1\cap L^q)}.
$$
It remains to estimate the nonlinearity $f(u)$ in $L^r$ for $r=1,\,q.$ To this end, let us remark that
\begin{equation}\label{3.4}
|f(u)|
\leq C|u|^m\left({\rm e}^{\lambda u^2}-1\right)+C|u|^m.
\end{equation}
By H\"{o}lder inequality and Lemma \ref{sarah55}, we have for $1\leq r\leq q$ and since $m\geq 2$,
\begin{eqnarray}\label{3.6}
\nonumber
\|f(u)\|_{{r}} &\leq & C\|u\|_{{mr}}^m +C\||u|^m({\rm e}^{\lambda u^2}-1)\|_{{r}}\\ & \leq & C\|u\|_{{mr}}^m +C\|u\|_{{2mr}}^m\|{\rm e}^{\lambda u^2}-1\|_{{2r}} \\ &\leq& \nonumber C\|u\|_{\EP}^m\bigg(\|{\rm e}^{\lambda u^2}-1\|_{{2r}}+1\bigg).
\end{eqnarray}
According to Lemma \ref{med},  and the fact that $u\in Y_M$, we have for $2q\lambda M^2\leq 1$,
\begin{equation}
\label{3.8}
\|f(u)\|_{L^\infty(0,\infty; L^{r})}\leq CM^m.
\end{equation}

 Finally, we obtain
\begin{eqnarray*}
\|\Phi(u)\|_{L^\infty(0,\infty, \EP)}&\leq& \mathcal{H} \|u_{0}\|_{\EP}+C M^m\\
&\leq& \mathcal{H}\varepsilon+CM^m.
\end{eqnarray*}

Let  $u,\,v$ be two elements of $Y_M.$ By using \eqref{taylorm} and Proposition \ref{LPLQQ}, we obtain
\begin{eqnarray*}
 t^{\sigma}\|\Phi(u)(t)-\Phi(v)(t)\|_{p}&\leq&
t^{\sigma}\int_{0}^{t}\left\| {\rm e}^{-(t-s)\Delta^2}
(f(u(s))-f(v(s)))\right\|_{p} ds\\&\leq& {\mathcal H} t^{\sigma}
\int_{0}^{t}(t-s)^{-{N\over 4}(\frac{1}{r}-\frac{1}{p})}\left\|f(u(s))-f(v(s))\right\|_{r}\,ds\\
&&\hspace{-3cm}\leq {\mathcal H}C\sum_{k=0}^{\infty}\frac{\lambda^{k}}{k!} t^{\sigma}\int_{0}^{t}(t-s)^{-{N\over 4}(\frac{1}{r}-\frac{1}{p})}
\|(u-v)(|u|^{2k+m-1}+|v|^{2k+m-1})\|_{r}ds,
\end{eqnarray*}
where $1\leq r\leq p.$ We use the H\"{o}lder inequality with ${1\over r}={1\over p}+{1\over q}$ to obtain
\begin{eqnarray*}
 t^{\sigma}\|\Phi(u)(t)-\Phi(v)(t)\|_{p}&\leq&{\mathcal H} C\sum_{k=0}^{\infty}\frac{\lambda^{k}}{k!} t^{\sigma}\int_{0}^{t}(t-s)^{-{N\over 4}(\frac{1}{r}-\frac{1}{p})}
\|u-v\|_{p}\times\\
&&\||u|^{2k+m-1}+|v|^{2k+m-1}\|_{q}ds,\\
&\leq& {\mathcal H} C\sum_{k=0}^{\infty}\frac{\lambda^{k}}{k!} t^{\sigma}\int_{0}^{t}(t-s)^{-{N\over 4}(\frac{1}{r}-\frac{1}{p})}
\|u-v\|_{p}\times\\ &&\left(\|u\|^{2k+m-1}_{{q(2k+m-1)}}+\|v\|^{2k+m-1}_{{q(2k+m-1)}}\right)ds.
\end{eqnarray*}

Using interpolation inequality where $\frac{1}{q(2k+m-1)}=\frac{\theta}{p}+\frac{1-\theta}{\rho},\; 2\leq \rho<\infty,$  we have
\begin{equation*}
\begin{split}
t^{\sigma}\left\|\int_0^t{\rm e}^{-(t-s)\Delta^2}\,(f(u)-f(v))\,ds\right\|_{p}\,\,\!\!\!\!
&\leq \sum_{k=0}^{\infty}\frac{\lambda^{k}}{k!}t^{\sigma}\int_0^t(t-s)^{-\frac{N}{4}
(\frac{1}{r}-\frac{1}{p})}\|u-v\|_{p}\\
&\hspace{-3cm}\times\biggl(\|u\|^{(2k+m-1)\theta}_{{p}}
  \|u\|^{(2k+m-1)(1-\theta)}_{{\rho}}
  +\|v\|^{(2k+m-1)\theta}_{{p}}\|v\|^{(2k+m-1)(1-\theta)}_{{\rho}}\biggr)\,ds.
\end{split}
\end{equation*}
By Lemma \ref{sarah55}, we obtain
\begin{eqnarray}\label{4.16b}
&&t^{\sigma}\left\|\int_0^t{\rm e}^{-(t-s)\Delta^2}\,(f(u)-f(v))\,ds\right\|_{p}\nonumber\\
&&\qquad\qquad\leq C\sum_{k=0}^{\infty}\frac{\lambda^{k}}{k!}t^{\sigma}
\int_0^t(t-s)^{-\frac{N}{4}(\frac{1}{r}-\frac{1}{p})}\|u-v\|_{p}\Gamma\left(\frac{\rho}{2}+1\right)
^{\frac{(2k+m-1)(1-\theta)}{\rho}}\qquad\qquad\nonumber\\
&&\qquad\times\left(\|u\|^{(2k+m-1)\theta}_{{p}}\|u\|^{(2k+m-1)(1-\theta)}_{\EP}+
\|v\|^{(2k+m-1)\theta}_{{p}}\|v\|^{(2k+m-1)(1-\theta)}_{\EP}\right)\,ds.
\end{eqnarray}
Applying the fact that $u,\; v\in\; Y_M $ in  \eqref{4.16b}, we see that
\begin{eqnarray}\label{4.18sb}
&&t^{\sigma}\left\|\int_0^t{\rm e}^{-(t-s)\Delta^2}\,(f(u)-f(v))\,ds\right\|_{p}\nonumber\\
&&\qquad\leq Cd(u,v)\sum_{k=0}^{\infty}\frac{\lambda^{k}}{k!}\Gamma\left(\frac{\rho}{2}+1\right)^{\frac{(2k+m-1)(1-\theta)}{\rho}} M^{2k+m-1}
\nonumber\\
&&\qquad \qquad\qquad\times t^{\sigma}\bigg(\int_0^t(t-s)^{-\frac{N}{4}(\frac{1}{r}-\frac{1}{p})}s^{-\sigma(1+(2k+m-1)\theta)}\,ds\bigg)\nonumber\\
&&\qquad\leq Cd(u,v)\,\sum_{k=0}^{\infty}\frac{\lambda^{k}}{k!}\Gamma\left(\frac{\rho}{2}+1\right)^{\frac{(2k+m-1)(1-\theta)}{\rho}} M^{2k+m-1}
\qquad\qquad\qquad\qquad\qquad\nonumber\\
&&\qquad\qquad\qquad\times {\mathcal{B}}\left(1-\frac{N}{4}\left(\frac{1}{r}-\frac{1}{p}\right),1-\sigma\big(1+(2k+m-1)\theta\big)\right),
\end{eqnarray}
where the exponents $p,\,q,\,r,\,\theta,\,\rho$  satisfy for all $k\geq 0,$
\begin{eqnarray*}
 p>{N(m-1)\over 4}, \quad 1\leq r\leq p, && \frac{N}{4}\left(\frac{1}{r}-\frac{1}{p}\right)<1,\quad \sigma\big(1+(2k+m-1)\theta\big)<1,\nonumber\\
  0\leq\theta=\theta_k\leq1,\quad \frac{1}{r}=\frac{1}{p}+\frac{1}{q}, &&\quad \frac{1}{(2k+m-1)q}=\frac{\theta}{p}+\frac{1-\theta}{\rho},\quad 2\leq \rho=\rho_k<\infty,\nonumber\\
  &&\hspace{-3cm}1-\frac{N}{4}\left(\frac{1}{r}-\frac{1}{p}\right)-
  (2k+m-1)\theta\sigma=0.
\end{eqnarray*}
For any $p>{N(m-1)\over 4}$, one can choose $0<\theta_k<\frac{1}{2k+m-1}\min\left(m-1,\,{1-\sigma\over \sigma}\right).$ It is obvious that for such $\theta_k$, there exist $r,\,q,\,\rho$ such that the above conditions are satisfied.
For these parameters, using \eqref{gamma4} and \eqref{gamma1}, we obtain that
\begin{eqnarray}\label{4.20b}
\nonumber
 &&\hspace{-3cm}{\mathcal{B}}\left(1-\frac{N}{4}\left(\frac{1}{r}-\frac{1}{p}\right),1-\sigma\big(1+(2k+m-1)\theta\big)\right)=\\ && {\Gamma\left({m-2\over m-1}+{N\over 4p}\right)\over \Gamma\left(1-\frac{N}{4}(\frac{1}{r}-\frac{1}{p})\right)\Gamma\Bigl(1-\sigma[1+(2k+m-1)\theta]\Bigr)}\leq C .
\end{eqnarray}
Moreover, note that $\theta_k\to 0,\; \rho_k\to \infty$ and
$$\frac{(2k+m-1)(1-\theta_k)}{2\rho_k}(1+\rho_k)\leq k,\; \forall\; k\geq 1.$$
This together with \eqref{gamma3} and \eqref{gamma2} gives
\begin{equation}\label{4.21b}
   \Gamma\left(\frac{\rho_k}{2}+1\right)^{\frac{(2k+m-1)(1-\theta_k)}{\rho_k}} \leq C^k k !.
\end{equation}
 Combining \eqref{4.18sb}, \eqref{4.20b} and \eqref{4.21b} we have
$$
t^{\sigma}\left\|\int_0^t{\rm e}^{-(t-s)\Delta^2}\,(f(u)-f(v))\,ds\right\|_{p}\leq C d(u,v) \sum_{k=0}^{\infty} \,{(C\lambda)^k} M^{2k+m-1}.$$
Then, we get for $M$ small,
\begin{equation*}\label{4.18sss}
t^{\sigma}\left\|\int_0^t{\rm e}^{-(t-s)\Delta^2}\,(f(u)-f(v))\,ds\right\|_{p}\leq C M^{m-1} d(u,v).
\end{equation*}

The above estimates show that $\Phi : Y_M \to Y_M $ is a contraction mapping.  By Banach's fixed point theorem, we thus obtain the existence of a unique
 $u$ in  $Y_M$ with $\Phi(u)=u.$ By \eqref{Phi}, $u$ solves the integral equation \eqref{NN9} with $f$ satisfying \eqref{NNN9}. The estimate \eqref{Linfinibihar} follows from $u\in Y_M.$ This terminates the proof of the existence of a  global solution to \eqref{NN9} for $N\geq 9$.

\subsection{The case $N=8$}
Let  $u,\,v$ be two elements of $Y_M.$ By using \eqref{taylorm} and Proposition \ref{LPLQQ}, we obtain
\begin{eqnarray*}
 t^{\sigma}\|\Phi(u)(t)-\Phi(v)(t)\|_{p}&\leq&
t^{\sigma}\int_{0}^{t}\left\| {\rm e}^{-(t-s)\Delta^2}
(f(u(s))-f(v(s)))\right\|_{p} ds\\&\leq& {\mathcal H} t^{\sigma}
\int_{0}^{t}(t-s)^{-2(\frac{1}{r}-\frac{1}{p})}\left\|f(u(s))-f(v(s))\right\|_{r}\,ds\\
&&\hspace{-3cm}\leq {\mathcal H}C\sum_{k=0}^{\infty}\frac{\lambda^{k}}{k!} t^{\sigma}\int_{0}^{t}(t-s)^{-2(\frac{1}{r}-\frac{1}{p})}
\|(u-v)(|u|^{2k+m-1}+|v|^{2k+m-1})\|_{r}ds,
\end{eqnarray*}
where $1\leq r\leq p.$ We use the H\"{o}lder inequality with ${1\over r}={1\over p}+{1\over q}$ to obtain
\begin{eqnarray*}
 t^{\sigma}\|\Phi(u)(t)-\Phi(v)(t)\|_{p}&\leq&{\mathcal H} C\sum_{k=0}^{\infty}\frac{\lambda^{k}}{k!} t^{\sigma}\int_{0}^{t}(t-s)^{-2(\frac{1}{r}-\frac{1}{p})}
\|u-v\|_{p}\times\\
&&\||u|^{2k+m-1}+|v|^{2k+m-1}\|_{q}ds,\\
&\leq& {\mathcal H} C\sum_{k=0}^{\infty}\frac{\lambda^{k}}{k!} t^{\sigma}\int_{0}^{t}(t-s)^{-2(\frac{1}{r}-\frac{1}{p})}
\|u-v\|_{p}\times\\ &&\left(\|u\|^{2k+m-1}_{{q(2k+m-1)}}+\|v\|^{2k+m-1}_{{q(2k+m-1)}}\right)ds.
\end{eqnarray*}
Similar calculations as in the previous subsection give
\begin{equation}\label{4.12}
t^{\sigma}\|\Phi(u)-\Phi(v)\|_{p}\leq  CM^{m-1}\sup_{s>0}s^{\sigma}\|u-v\|_{p}=CM^{m-1} d(u,v).
\end{equation}
We now estimate $\|\Phi(u)\|_{L^{\infty}(0,\infty;\exp L^2(\R^8))}.$ We have, by Proposition \ref{pre},
\begin{equation*}
\begin{split}
\|\Phi(u)\|_{L^{\infty}(0,\infty;\exp L^2(\R^8))}&\leq \|{\rm e}^{-t\Delta^2}u_{0}\|_{L^{\infty}(0,\infty;\exp L^2(\R^8))}\\
&+\left\|\int_{0}^{t}{\rm e}^{-(t-s)\Delta^2} (f(u(s))) ds\right\|_{L^{\infty}(0,\infty;\exp L^2(\R^8))}\\ &\leq {\mathcal H}\|u_{0}\|_{\exp L^2(\R^8))}\\
&+\left\|\int_{0}^{t}{\rm e}^{-(t-s)\Delta^2} (f(u(s))) ds\right\|_{L^{\infty}(0,\infty;\exp L^2(\R^8))}.
\end{split}
\end{equation*}
We first estimate $\left\|\displaystyle\int_0^t{\rm e}^{-(t-s)\Delta^2}(f(u(s)))\,ds\right\|_{L^{\infty}(0,\infty;\exp L^2(\R^8))}.$ Using \eqref{slim}, it suffice to estimate  $\left\|\displaystyle\int_0^t{\rm e}^{-(t-s)\Delta^2}(f(u(s)))\,ds\right\|_{L^{\infty}(0,\infty;L^{\phi}(\R^8))}$ and $\left\|\displaystyle\int_0^t{\rm e}^{-(t-s)\Delta^2}(f(u(s)))\,ds\right\|_{L^{\infty}(0,\infty;L^{2}(\R^8))}.$

By the same argument as in the case $N\geqslant9,$  using Corollary \ref{forn=8} instead of Corollary \ref{lemme estimates} we obtain
\begin{equation}
\label{p}
\begin{split}
\left\|\int_{0}^{t}{\rm e}^{-(t-s)\Delta^2} (f(u(s)))\,ds\right\|_{L^{\infty}(0,\infty;L^{\phi}(\R^8))} &\leq C M^m.
\end{split}
\end{equation}
Second  we estimate $\left\|\displaystyle\int_0^t{\rm e}^{-(t-s)\Delta^2}\,(f(u(s)))\, ds\right\|_{L^{\infty}(0,\infty;L^2(\R^8))}.$
By using \eqref{taylorm} and Proposition \ref{LPLQQ}, we obtain
\begin{equation*}
\begin{split}
&\left\|\int_{0}^{t}{\rm e}^{-(t-s)\Delta^2}(f(u(s))) ds\right\|_{2}\leq {\mathcal H}C\sum_{k=0}^{\infty}\frac{\lambda^k}{k!}\int_{0}^{t}(t-s)^{-2(\frac{1}{r}-\frac{1}{2})}
\|u\|^{2k+m}_{{(2k+m)r}}\,ds.
\end{split}
\end{equation*}
Using similar computations as above, we obtain
\begin{equation}
\label{4.30}
 \left \|\int_0^t{\rm e}^{-(t-s)\Delta^2}\,(f(u(s))) ds\right\|_{L^{\infty}(0,\infty; L^2(\R^8))}\leq CM^m.
\end{equation}

From \eqref{p} and \eqref{4.30}, we obtain
\begin{equation*}
 \left \|\int_0^t{\rm e}^{-(t-s)\Delta^2}\,(f(u(s))) ds\right\|_{L^{\infty}(0,\infty;\exp L^2(\R^8))}\leq C M^m.
\end{equation*}
It follows that
\begin{equation*}
\begin{split}
\|\Phi(u)\|_{L^{\infty}(0,\infty,\exp L^2(\R^8))}&\leq {\mathcal H} \|u_{0}\|_{\exp L^2(\R^8))}+CM^m.
\end{split}
\end{equation*}

Now, by  \eqref{med3} the inequality \eqref{4.12} gives
\begin{equation*}
t^{\sigma}\|\Phi(u)\|_{p}\leq {\mathcal H}\|u_{0}\|_{\exp L^2(\R^8)}+ CM^m.
\end{equation*}
If we choose  $M$ and $\varepsilon$ small then  $\Phi$ maps $Y_M$ into itself.
Moreover, thanks to the inequality \eqref{4.12}  we obtain that  $\Phi$ is a contraction  map on $Y_M$. The conclusion follows by the Banach fixed point theorem.

\begin{remark}
{\rm We do not need the restriction $p<4$ unlike in \cite{Ioku}. Indeed, such a restriction comes from a particular choice of $\theta$ which can be avoid.}
\end{remark}


\subsection{The case of $N\leq 7$}

According to \eqref{med3} and the previous calculations, it remains to establish the following two inequalities
\begin{equation}
\label{estim1}
\left\|\int_{0}^{t}{\rm e}^{-(t-s)\Delta^2}(f(u(s))) ds\right\|_{L^{\infty}(0,\infty;\,\EP)}\leq C_1(M),
\end{equation}
and
 \begin{equation}
\label{estim2}
\sup_{t>0}t^{\sigma}\left\|\int_0^t{\rm e}^{-(t-s)\Delta^2}\,(f(u)-f(v))\,ds\right\|_{p}\leq C_2(M)\sup_{s>0}s^{\sigma}\|u(s)-v(s)\|_{p},
\end{equation}
 where $u,\;v\in Y_M$ and with $C_1$ and $C_2$ are small when $M$ is small.\\

\noindent{\bf Estimate \eqref{estim1}}.  We have
\begin{equation}\label{log}
    \left(\log\left((t-s)^{-N/4}+1\right)\right)^{-\frac{1}{2}}\leq\sqrt{2}(t-s)^{\frac{N}{8}}\quad \mbox{for}\quad 0\leq s\leq t-a^{-\frac{4}{N}},
\end{equation}
where $a>1$ is the number satisfying $a=2\log(a+1).$  Therefore we have, for ${N\over 4}<q\leq 2$ and $0<t\leq a^{-4/N}$,
\begin{eqnarray*}
\left\|\int_0^t{\rm e}^{-(t-s)\Delta^2}\,f(u(s))\,ds\right\|_{\EP}&\leq&{\mathcal H}\int_0^t(t-s)^{-\frac{N}{4q}}(\log((t-s)^{-\frac{N}{4}}+1))^{-\frac{1}{2}}\|f(u(s))\|_{q}\,ds\\
&\leq&C\sup_{t>0}\|f(u(t))\|_{q},
\end{eqnarray*}
where here
$$
C={\mathcal H}\int_0^{a^{-4/N}}\tau^{-\frac{N}{4q}}\,\left(\log\left(\tau^{-\frac{N}{4}}+1\right)\right)^{-\frac{1}{2}}\,d\tau<\infty.
$$
For $t\geq a^{-4/N}$, we write
\begin{eqnarray}\label{(4.31)}
&&\qquad\qquad\left\|\int_0^t{\rm e}^{-(t-s)\Delta^2}\,f(u(s))\,ds\right\|_{\EP}\nonumber\\
&&\qquad\qquad\leq {\mathcal H} \int_0^{t-a^{-\frac{4}{N}}}(t-s)^{-\frac{N}{4q}}(\log((t-s)^{-\frac{N}{4}}+1))^{-\frac{1}{2}}\|f(u(s))\|_{q}\,ds
\qquad\qquad\qquad\\ \nonumber
&&\qquad\qquad+{\mathcal H}
\int_{t-a^{-\frac{4}{N}}}^t(t-s)^{-\frac{N}{4q}}(\log((t-s)^{-\frac{N}{4}}+1))^{-\frac{1}{2}}\|f(u(s))\|_{q}\,ds\\ \nonumber
&&\qquad\qquad\leq{\mathcal H}\sqrt{2}\int_0^t(t-s)^{-\frac{N}{4q}+{\frac{N}{8}}}\|f(u(s))\|_{q}\,ds+C\sup_{t>0}\|f(u(t))\|_{q}\nonumber= \textbf{I}+\textbf{J}.
\end{eqnarray}
We first estimate $\textbf{I}$. By \eqref{taylorm} and the fact that $f(0)=0$,  we have
\begin{equation*}
\textbf{I}\leq {\mathcal H}C\sqrt{2}\sum_{k=0}^{\infty}\frac{\lambda^k}{k!}\int_0^t(t-s)^{-\frac{N}{4q}+\frac{N}{8}}\|u\|_{{(2k+m)q}}^{2k+m}\,\,ds.
\end{equation*}
Using interpolation inequality and Lemma \ref{sarah55}, we get
\begin{eqnarray*}
\textbf{I}&\leq&C\sum_{k=0}^{\infty}\frac{\lambda^k}{k!}\int_0^t(t-s)^{-\frac{N}{4q}+\frac{N}{8}}\|u\|_{p}^{(2k+m)
\theta}\|u\|_{\rho}^{(2k+m)(1-\theta)}\,ds\nonumber\\&\leq&
C\sum_{k=0}^{\infty}\frac{\lambda^k}{k!}\int_0^t(t-s)^{-\frac{N}{4q}+\frac{N}{8}}\|u\|_{p}^{(2k+m)
\theta}\Gamma\left(\frac{\rho}{2}+1\right)
^{\frac{(2k+m)(1-\theta)}{\rho}}\|u\|_{\EP}^{(2k+m)(1-\theta)}\,ds\nonumber\\
&\leq&C\sum_{k=0}^{\infty}\frac{\lambda^k}{k!}\Gamma\left(\frac{\rho}{2}+1\right)
^{\frac{(2k+m)(1-\theta)}{\rho}}M^{2k+m}\int_0^t(t-s)^{-\frac{N}{4q}+\frac{N}{8}}s^{-(2k+m)\theta\sigma}\,ds
\nonumber\\
&\leq&C\sum_{k=0}^{\infty}\frac{\lambda^k}{k!}\Gamma\left(\frac{\rho}{2}+1\right)
^{\frac{(2k+m)(1-\theta)}{\rho}}M^{2k+m}t^{1-\frac{N}{4q}+\frac{N}{8}-(2k+m)\theta\sigma}\\
&&\times {\mathcal{B}}\left(1-\frac{N}{4q}+\frac{N}{8},1-(2k+m)\theta\sigma\right),
\end{eqnarray*}
where ${\mathcal{B}}$ is the beta function and $\rho,\,\theta,\,q$  satisfy, for all $k,$
\begin{eqnarray*}
p>{2N(m-1)\over 8-N},\quad 0\leq \theta=\theta_k\leq1,\quad  \frac{N}{4} <q<2,\quad \frac{N}{4q}-{\frac{N}{8}}<1,\quad && (2k+m)\theta\sigma<1, \\
 \quad 1-\frac{N}{4q}+\frac{N}{8}-(2k+m)\theta\sigma=0, \quad \frac{1}{(2k+m)q}=\frac{\theta}{p}+\frac{1-\theta}{\rho},\quad &&\;  2\leq \rho=\rho_k<\infty.
\end{eqnarray*}
For any $p>{2N(m-1)\over 8-N}$, one can choose $\frac{N}{8(2k+m)\sigma}<\theta_k<\frac{m-1}{2k+m}.$ It is obvious that for such $\theta_k$, there exist $q,\,\rho$ such that the above conditions are satisfied.

Arguing as above, we obtain
\begin{equation}\label{(4.34)}
  {\mathcal{B}}\left(1-\frac{N}{4q}+\frac{N}{8},1-(2k+m)\theta\sigma\right)={1\over \Gamma\left(1-\frac{N}{4q}+\frac{N}{8}\right) \Gamma\Big(1-(2k+m)\theta\sigma\Big)}\leq C,
\end{equation}
and
\begin{equation}\label{(4.35)}
  \Gamma\left(\frac{\rho_k}{2}+1\right)^{\frac{(2k+m)(1-\theta_k)}{\rho_k}} \leq C^k k!,
\end{equation}

Combining  \eqref{(4.34)} and \eqref{(4.35)}, we have, for small M,
\begin{equation}\label{(4.36)}
  \textbf{I}\leq C\,M^m.
\end{equation}

To estimate the term $\textbf{J}$, we write
\begin{equation*}
    \|f(u)\|_{q}\leq C\||u|^m {\rm e}^{\lambda u^2}\|_{q}\leq C\||u|^m({\rm e}^{\lambda u^2}+1-1)\|_{q}.
\end{equation*}
By H\"older inequality, we obtain
\begin{eqnarray*}
   \|f(u)\|_{q}&\leq& C\|u\|^m_{{2mq}}\|{\rm e}^{\lambda u^2}-1\|_{{2q}}+\|u(t)\|^m_{{mq}}\\
&\leq&C\|u\|^m_{\EP}\left(\|{\rm e}^{\lambda u^2}-1\|_{{2q}}+1\right),
\end{eqnarray*}
where we have used $mq>Nm/4>N(m-1)/4\geq 2.$ Now, by Lemma \ref{med}, for $2q\lambda M^2\leq 1,$ we have
\begin{equation*}
\|{\rm e}^{\lambda u^2}-1)\|_{{2q}}\leq (2q\lambda)^{\frac{1}{2q}} M^{\frac{1}{q}}\leq 1.
\end{equation*}
Then we conclude that, for $u\in Y_M,$
\begin{equation*}
   \textbf{J}=\sup_{t>0}\|f(u(t))\|_{q}\leq C M^m.
\end{equation*}
Finally, we obtain
\begin{eqnarray}\label{(4.991)}
\left\|\int_0^t{\rm e}^{-(t-s)\Delta^2}\,(f(u))\,ds\right\|_{\EP
}&\leq&CM^m.
\end{eqnarray}

{\bf Estimate \eqref{estim2}}. By \eqref{taylorm} and Proposition \ref{LPLQQ}, we have
\begin{equation*}
\begin{split}
&t^{\sigma}\left\|\int_0^t{\rm e}^{-(t-s)\Delta^2}\,(f(u)-f(v))\,ds\right\|_{p}\\
&\qquad\qquad\qquad\leq C\, \sum_{k=0}^{\infty}\frac{\lambda^k}{k!}t^{\sigma}\int_0^t
(t-s)^{-\frac{N}{4}(\frac{1}{r}-\frac{1}{p})}\|(u-v)(|u|^{2k+m-1}+|v|^{2k+m-1})\|_{r}\,ds.\qquad\qquad
\end{split}
\end{equation*}
Applying the H\"older's inequality, we obtain
\begin{eqnarray*}
&& t^{\sigma}\left\|\int_0^t{\rm e}^{-(t-s)\Delta^2}\,(f(u)-f(v))\,ds\right\|_{p}\\
&&\quad\leq C\, \sum_{k=0}^{\infty}\frac{\lambda^k}{k!}t^{\sigma}\int_0^t(t-s)^{-\frac{N}{4}(\frac{1}{r}-\frac{1}{p})}\|u-v\|_{p} \|(|u|^{2k+m-1}+|v|^{2k+m-1})\|_{q}\,ds\\
&&\quad\leq C\, \sum_{k=0}^{\infty}\frac{\lambda^k}{k!}t^{\sigma}\int_0^t(t-s)^{-\frac{N}{4}
(\frac{1}{r}-\frac{1}{p})}\|u-v\|_{p}\left(\|u\|^{2k+m-1}_{{q(2k+m-1)}}+\|v\|^{2k+m-1}_{{q(2k+m-1)}}\right)\,ds.
\end{eqnarray*}
Using interpolation inequality where $\frac{1}{q(2k+m-1)}=\frac{\theta}{p}+\frac{1-\theta}{\rho},\; 2\leq \rho<\infty,$  we have
\begin{equation*}
\begin{split}
t^{\sigma}\left\|\int_0^t{\rm e}^{-(t-s)\Delta^2}\,(f(u)-f(v))\,ds\right\|_{p}\,\,\!\!\!\!
&\leq \sum_{k=0}^{\infty}\frac{\lambda^k}{k!}t^{\sigma}\int_0^t(t-s)^{-\frac{N}{4}
(\frac{1}{r}-\frac{1}{p})}\|u-v\|_{p}\\
&\hspace{-3cm}\times(\|u\|^{(2k+m-1)\theta}_{{p}}
  \|u\|^{(2k+m-1)(1-\theta)}_{{\rho}}
  +\|v\|^{(2k+m-1)\theta}_{{p}}\|v\|^{(2k+m-1)(1-\theta)}_{{\rho}})\,ds.
\end{split}
\end{equation*}
By Lemma \ref{sarah55}, we obtain
\begin{eqnarray}\label{4.16}
&&t^{\sigma}\left\|\int_0^t{\rm e}^{-(t-s)\Delta^2}\,(f(u)-f(v))\,ds\right\|_{p}\nonumber\\
&&\qquad\qquad\leq C\sum_{k=0}^{\infty}\frac{\lambda^k}{k!}t^{\sigma}
\int_0^t(t-s)^{-\frac{N}{4}(\frac{1}{r}-\frac{1}{p})}\|u-v\|_{p}\Gamma\left(\frac{\rho}{2}+1\right)
^{\frac{(2k+m-1)(1-\theta)}{\rho}}\qquad\qquad\nonumber\\
&&\quad\times\left(\|u\|^{(2k+m-1)\theta}_{{p}}\|u\|^{(2k+m-1)(1-\theta)}_{\EP}+
\|v\|^{(2k+m-1)\theta}_{{p}}\|v\|^{(2k+m-1)(1-\theta)}_{\EP}\right)\,ds.
\end{eqnarray}
Applying the fact that $u,\; v\in\; Y_M $ in  \eqref{4.16}, we see that
\begin{eqnarray}\label{4.18s}
&&t^{\sigma}\left\|\int_0^t{\rm e}^{-(t-s)\Delta^2}\,(f(u)-f(v))\,ds\right\|_{p}\nonumber\\
&&\qquad\leq Cd(u,v)\sum_{k=0}^{\infty}\frac{\lambda^k}{k!}\Gamma\left(\frac{\rho}{2}+1\right)^{\frac{(2k+m-1)(1-\theta)}{\rho}} M^{2k+m-1}
\nonumber\\
&&\qquad \qquad\qquad\times t^{\sigma}\bigg(\int_0^t(t-s)^{-\frac{N}{4}(\frac{1}{r}-\frac{1}{p})}s^{-\sigma(1+(2k+m-1)\theta)}\,ds\bigg)\nonumber\\
&&\qquad\leq Cd(u,v)\,\sum_{k=0}^{\infty}\frac{\lambda^k}{k!}\Gamma\left(\frac{\rho}{2}+1\right)^{\frac{(2k+m-1)(1-\theta)}{\rho}} M^{2k+m-1}
\qquad\qquad\qquad\qquad\qquad\nonumber\\
&&\qquad\qquad\qquad\times {\mathcal{B}}\left(1-\frac{N}{4}\left(\frac{1}{r}-\frac{1}{p}\right),1-\sigma(1+(2k+m-1)\theta)\right),
\end{eqnarray}
where the exponents $p,\,q,\,r,\,\theta,\,\rho$  satisfy for all $k,$
\begin{eqnarray*}
 p>m, \quad 1\leq r\leq p, && \frac{N}{4}\left(\frac{1}{r}-\frac{1}{p}\right)<1,\quad \sigma\big(1+(2k+m-1)\theta\big)<1,\nonumber\\
  0\leq\theta=\theta_k\leq1,\quad \frac{1}{r}=\frac{1}{p}+\frac{1}{q}, &&\quad \frac{1}{(2k+m-1)q}=\frac{\theta}{p}+\frac{1-\theta}{\rho},\quad 2\leq \rho<\infty,\nonumber\\
  &&\hspace{-3cm}1-\frac{N}{4}\left(\frac{1}{r}-\frac{1}{p}\right)-
  (2k+m-1)\theta\sigma=0.
\end{eqnarray*}
For any $p>m$, one can choose
$$\frac{{N\over 4p}+1-{N\over 4}}{(2k+m-1)\sigma}<\theta_k<\frac{1}{2k+m-1}\min\left(m-1,\,{1-\sigma\over \sigma}\right).$$ It is obvious that for such $\theta_k$, there exist $r,\,q,\,\rho$ such that the above conditions are satisfied.

Using \eqref{gamma4}, \eqref{gamma1} and the fact that $1-\sigma>0$ (even for $p=\infty$),  we obtain that
\begin{equation}\label{4.20}
 {\mathcal{B}}\left(1-\frac{N}{4}\left(\frac{1}{r}-\frac{1}{p}\right),1-\sigma\big(1+(2k+m-1)\theta\big)\right)\leq C.
\end{equation}
As above, we also have
\begin{equation}\label{4.21}
   \Gamma\left(\frac{\rho_k}{2}+1\right)^{\frac{(2k+m-1)(1-\theta_k)}{\rho_k}} \leq C^k k !.
\end{equation}
 Combining \eqref{4.18s}, \eqref{4.20} and \eqref{4.21} we have
\begin{eqnarray*}
&&t^{\sigma}\left\|\int_0^t{\rm e}^{-(t-s)\Delta^2}\,(f(u)-f(v))\,ds\right\|_{p}\nonumber\\
&&\qquad\qquad\leq C\,d(u,v)\,\sum_{k=0}^{\infty}\frac{\lambda^k}{k!}\Gamma\left(\frac{\rho}{2}+1\right)^{\frac{(2k+m-1)(1-\theta)}{\rho}} M^{2k+m-1}\nonumber\\
&&\qquad\qquad\qquad\times {\mathcal B}\left(1-\frac{N}{4}\left(\frac{1}{r}-\frac{1}{p}\right),1-\sigma(1+(2k+m-1)\theta)\right)\nonumber\\
&&\qquad\qquad\leq C d(u,v) \sum_{k=0}^{\infty}{(C\lambda)^k} M^{2k+m-1}.\nonumber
\end{eqnarray*}
Then, we get
\begin{equation*}\label{4.18ssss}
t^{\sigma}\left\|\int_0^t{\rm e}^{-(t-s)\Delta^2}\,(f(u)-f(v))\,ds\right\|_{p}\leq C M^{m-1} d(u,v).
\end{equation*}
This together with \eqref{(4.991)} and \eqref{med3} concludes the proof of global existence for dimensions $N\leq 7$.


\subsection{Continuity at zero}

We will now prove the statement \eqref{1.8}. For ${ q>\frac{N}{4}}$, we have
\begin{eqnarray}\label{V.549}
 && \|u(t)-{\rm e}^{-t\Delta^2}u_0\|_{\EP} \qquad\qquad\qquad\qquad\nonumber\\
 &&\qquad\leq \int_0^t\|{\rm e}^{-(t-s)\Delta^2}
  f(u(s))\|_{\EP}\,ds\nonumber\\
  &&\qquad\leq {{1\over \sqrt{\log 2}}}\int_0^t\|{\rm e}^{-(t-s)\Delta^2}f(u(s))\|_{2}ds+{{1\over \sqrt{\log 2}}}\int_0^t\|{\rm e}^{-(t-s)\Delta^2}f(u(s))\|_{{\infty}}\,ds\nonumber\\
  &&\qquad\leq {{{\mathcal{H}}\over \sqrt{\log 2}}}\int_0^t\|f(u(s))\|_{2}ds+{{{\mathcal{H}}\over \sqrt{\log 2}}}\int_0^t(t-s)^{-\frac{N}{4q}} \|f(u(s))\|_{q}\,ds.
\end{eqnarray}
Now, let us  estimate $\|f(u(t))\|_{r}$ for $r=2,\,q.$  We have
\begin{equation*}
    |f(u)|\leq C|u|^m{\rm e}^{\lambda u^2}.
\end{equation*}
Therefore, we obtain
\begin{equation*}
 \|f(u)\|_{r}\leq C \||u|^m({\rm e}^{\lambda u^2}-1+1)\|_{r}.
\end{equation*}
By using H\"{o}lder inequality and Lemma \ref{sarah55}, we obtain
\begin{eqnarray*}
  \|f(u)\|_{r}&\leq& C\|u\|^{m}_{{2mr}}\|{\rm e}^{\lambda u^2}-1\|_{{2r}}+\|u\|^{m}_{{mr}}\\
&\leq&C\|u\|^{m}_{\EP}\left(\|{\rm e}^{\lambda u^2}-1\|_{{2r}}+1\right).
\end{eqnarray*}
Using Lemma \ref{med} we conclude that
\begin{equation}\label{V.559}
\|f(u)\|_{r}\leq C \|u\|^{m}_{\EP} \left((2r\lambda)^{\frac{1}{2r}}M^{\frac{1}{r}}+1\right)\leq C \|u\|^{m}_{\EP}.
\end{equation}
Substituting \eqref{V.559} in \eqref{V.549}, we have
\begin{eqnarray*}
\|u(t)-{\rm e}^{-t\Delta^2}u_0\|_{\EP}&\leq& C\int_0^t \bigg[\|u\|^{m}_{\EP}+(t-s)^{-\frac{N}{4q}}
\|u\|^{m}_{\EP}\bigg]ds\\
&\leq& Ct \|u\|^{m}_{L^{\infty}(0,\infty;\,\EP)}+Ct^{1-\frac{N}{4q}}\|u\|^{m}_{L^{\infty}(0,\infty;\,\EP)} \\
&\leq &C_1t+C_2t^{1-\frac{N}{4q}},
\end{eqnarray*}
where $C_1,\,C_2$ are  finite positive constants. This gives
\begin{equation*}
    \lim_{t\longrightarrow0}\|u(t)-{\rm e}^{-t\Delta^2}u_0\|_{\EP}=0,
\end{equation*}
and proves statement \eqref{1.8}.

Finally the fact that $u(t)\to u_0$ as $t\to 0$ in the weak$^*$ topology can be done as in \cite{Ioku}. So we omit the proof here. This completes the proof of Theorem \ref{GE}.
\subsection{Proof of Theorem \ref{linfinheat}}  The proof is via  a fixed point argument on the associated integral equation
 \begin{equation}\label{NN9H}
   u(t)= {\rm e}^{t\Delta}u_{0}+\int_{0}^{t}{\rm e}^{(t-s)\Delta}f(u(s))\,ds,
\end{equation}
 where  $\|u_0\|_{\EP}\leq \varepsilon$, with  small $\varepsilon>0$.   For $M>0$ we  use the space
$$Y_M :=\left\{u\in L^\infty(0,\infty, \EP);\;\displaystyle\sup_{t>0}  t^{\sigma}\|u(t)\|_{p}+\|u\|_{L^{\infty}(0,\infty;\EP)}\leq M\right\},$$ endowed with  the metric $$d(u,v)=\displaystyle\sup_{t>0}  t^{\sigma}\|u(t)-v(t)\|_{p}.$$  Here $p>{N(\ell-1)\over 2}$ and $$\sigma={1\over \ell-1}-\frac{N}{2p}>0.$$ The rest of the proof is carried out as in the previous subsections, so we omit the details.
\section{Extensions to polyharmonic heat equations}
Our results can be extended to the nonlinear polyharmonic heat equations

\begin{equation}\label{1.1poly}
\left\{\begin{array}{cc}
\partial_{t} u +(-\Delta)^d u=f(u)  ,   \\
u(0,x)=u_{0}(x),
\end{array}
\right.
\end{equation}
where $u(t,x)$ is a real valued function  $t>0,\,\,x\in\R^N,$ $(-\Delta)^d,\; d\geq 2$, is the polyharmonic operator, and $f : \R\to\R$ with $f(0)=0$ and having an exponential growth at infinity.

In fact, the proofs can be obtained by slightly modifying the arguments used for $d=2.$ Precisely, using the majorizing kernel established in \cite{GP} and since the fundamental solution $E_d$ of
$$
\partial_{t} u +(-\Delta)^d u=0,\; t>0,\; x\in \R^N,
$$
satisfies
$$
E_d(t,x)={t^{-\frac{N}{2d}}}\,E_d\left(1,\,t^{-\frac{1}{2d}} x \right),
$$
we have the following $L^p-L^q$ estimate.
\begin{proposition}\label{LPLQQd} There exists a positive constant $ \mathcal{H}_d$ such that for all $1\leq p\leq q\leq\infty$, we have
\begin{equation*}
\|{\rm e}^{-t(-\Delta)^d}\varphi\|_{q}\leqslant \mathcal{H}_d  t^{-\frac{N}{2d}(\frac{1}{p}-\frac{1}{q})}\|\varphi\|_{p},\qquad\,\, \forall\; t>0,\,\, \forall\; \varphi\in L^p.
\end{equation*}
\end{proposition}
For the global existence the nonlinearity $f$ will satisfy \eqref{1.4} with $m$ a real number such that ${N(m-1)\over 2d}\geq 2,\; m\geq 2.$  Using similar arguments as in the proof of Theorem \ref{GE}, we obtain the following result.
\begin{theorem}[{Global existence}]
\label{GEd} Assume that the nonlinearity $f$ satisfies \eqref{1.4} with $m$ a real number such that ${N(m-1)\over 2d}\geq 2,$ $m\geq 2.$  Then, there exists a positive constant $\varepsilon>0$ such that for every initial data $u_0\in \EP$
which satisfies $\|u_0\|_{\EP} \leqslant\varepsilon,$ there exists a weak-mild solution $u\in L^{\infty}(0,\infty;\EP)$
of the Cauchy problem \eqref{1.1poly} satisfying
\begin{equation*}
    \lim_{t\longrightarrow0}\|u(t)-{\rm e}^{-t(-\Delta)^d}u_{0}\|_{\EP}=0.
\end{equation*}
Moreover, there exists a constant $C>0$ such that,
\begin{equation*}
\|u(t)\|_{p}\leq\,C\,t^{-\sigma},\quad\forall\;t>0,
\end{equation*}
where $$\sigma={1\over m-1}-{N\over 2dp}>0,$$with ${N(m-1)\over 2d}<p<\infty$ if $N\geq 4d$ and $\max\left\{m,\;{2N(m-1)\over 4d-N}\right\}\leq p\leq \infty$ if $N<4d$.
\end{theorem}


\begin{thebibliography}{10}
\bibitem{Adams}{R. A. Adams and J. J. F. Fournier}, ``Sobolev Spaces'', Second edition, Series Pure and Applied Mathematics (Amsterdam), Vol. 140, Elsevier, Academic Press, Amsterdam, 2003.

\bibitem{AsaiGiga}{T. Asai and Y. Giga}, {\em On self-similar solutions to the surface diffusion flow equations with contact angle boundary conditions}, Interfaces Free Bound., Vol. 16 (2014), 539--573.

\bibitem{BellShubinStephens}{J.B. Bell, G. R. Shubin and A.B. Stephens, {\em A segmentation approach to grid generation using biharmonics}, J. Comput. Phys., Vol. 47 (1982), 463--472.}

\bibitem{Bou}{J. Bourgain}, {\em Refinements of {S}trichartz' inequality and applications to
              {$2$}{D}-{NLS} with critical nonlinearity}, {Internat. Math. Res. Notices}, ({1998}), {253--283}.

\bibitem{CW}{T. Cazenave and F. B. Weissler}, {\em Asymptotically self-similar global solutions of the nonlinear {S}chr\"odinger and heat equations}, {Math. Z.}, Vol. 228 ({1998}), {83--120}.

\bibitem{egorov} {Y. V. Egorov, V. A.  Galaktionov, V. A. Kondratiev and S. I. Pohozaev}, {\em On the necessary conditions of global existence to a quasilinear inequality in the half-space},
{C. R. Acad. Sci. Paris S\'er. I Math.}, Vol. 330 (2000), {93--98}.

\bibitem{Escudero}{C. Escudero, F. Gazzola and I. Peral}, {\em Global existence versus blow-up results for a fourth order parabolic {PDE} involving the {H}essian}, {J. Math. Pures Appl. (9)}, Vol. 103 (2015), 924--957.

\bibitem{GP} {V. A. Galaktionov and S. I. Pohozaev}, {\em Existence and blow-up for higher-order semilinear parabolic equations: majorizing order-preserving operators}, {Indiana Univ. Math. J.}, Vol. 51 (2002), 1321--1338.

\bibitem{GalakMitidieri}{V. A. Galaktionov, E. Mitidieri and S. I. Pohozaev,  {\em Global sign-changing solutions of a higher order semilinear heat equation in the subcritical {F}ujita range}, {Adv. Nonlinear Stud.},
Vol. 12 (2012),569--596.}

\bibitem{Gazz} F. Gazzola, {\em On the moments of solutions to linear parabolic equation involving the biharmonic operator}, {Discrete Contin. Dyn. Syst.}, Vol. 33 (2013), {3583--3597}.

\bibitem{APDE}{S. Ibrahim, M. Majdoub and N. Masmoudi}, {\em Well- and ill-posedness issues for energy supercritical waves}, {Anal. PDE.}, Vol. 4 (2011), {341--367}.

\bibitem{IJMS} {S. Ibrahim, R. Jrad, M. Majdoub and T. Saanouni}, {\em Local well posedness of a 2{D} semilinear heat equation}, {Bull. Belg. Math. Soc. Simon Stevin}, Vol. 21 (2014), {535--551}.

\bibitem{Ioku} N. Ioku, {\em The Cauchy problem for heat equations with exponential nonlinearity}, J. Differential Equations, Vol. 251 (2011), 1172--1194.

\bibitem{IRT} N. Ioku, B. Ruf and E. Terraneo, {\em Existence, Non-existence, and Uniqueness for a Heat Equation with Exponential Nonlinearity in $\R^2$}, {Math. Phys. Anal. Geom.},  Vol. 18 (2015), Art. 29, 19 pp.

\bibitem{IKK} K. Ishige, T. Kawakami and K. Kobayashi, {\em Asymptotics for a nonlinear integral equation with a generalized heat kernel}, {J. Evol. Equ.}, Vol. 14 (2014), {749--777}.


\bibitem{Kenig} {C. E. Kenig, G. Ponce and L. Vega}, {\em Global well-posedness for semi-linear wave equations}, {Comm. Partial Differential Equations}, Vol. 25 (2000),{1741--1752}.

\bibitem{KingWinkler}{O.S.B. King, M. Winkler, {\em A fourth-order parabolic equation modeling epitaxial thin film growth}, J. Math. Anal. Appl., Vol. 286 (2003) 459--490.}


\bibitem{RR} M. M. Rao and Z. D. Ren, ``Applications of Orlicz spaces'', Monographs and Textbooks in Pure and Applied Mathematics, Marcel Dekker, Inc., New York, 2002.


\bibitem{RT} B. Ruf and E. Terraneo, {\em The Cauchy problem for a semilinear heat equation with singular initial data}, {Evolution equations, semigroups and functional analysis({M}ilano, 2000)},
{Progr. Nonlinear Differential Equations Appl.}, {295--309}, {Birkh\"auser, Basel}, {2002}.

\bibitem{SMG}{A. N. Sandjo, S. Moutari and Y. Gningue, {\em Solutions of fourth-order parabolic equation modeling thin-film growth}, J. Differential Equations, Vol. 259 (2015), 7260--7283.}

\bibitem{Trud}N. S. Trudinger, {\em On imbeddings into Orlicz spaces and some applications}, J. Math. Mech., Vol. 17 (1967), {473--483}.

\bibitem{Vertman}{B. Vertman, {\em The biharmonic heat operator on edge manifolds and non-linear fourth order equations}, {Manuscripta Math.}, Vol. 149 (2016),179--203.}

\bibitem{W}{F. B. Weissler, {\em Semilinear evolution equations in Banach spaces}, J. Funct. Anal., Vol. 32 (1979),  277--296.}

\bibitem{Indiana}{F. B. Weissler,  {\em Local existence and nonexistence for semilinear parabolic equations in {$L^{p}$}}, Indiana Univ. Math. J., Vol. 29 (1980), 79--102},

\end{thebibliography}
\end{document}